\documentclass{amsart}
\usepackage{amssymb, amsmath, amsthm, graphics, comment, xspace, enumerate}
\newtheorem{thm}{Theorem}[section]

\newtheorem{lem}[thm]{Lemma}
\newtheorem{prop}[thm]{Proposition}
\theoremstyle{definition}
\newtheorem{defn}[thm]{Definition}
\newtheorem{example}[thm]{Example}
\theoremstyle{remark}
\newtheorem{rem}[thm]{Remark}
\numberwithin{equation}{section}

\begin{document}
\title[Degenerate $C$-distribution cosine functions...]{Degenerate $C$-distribution cosine functions and
degenerate $C$-ultradistribution cosine functions
in locally convex spaces}

\author{Daniel Velinov}
\address{Department for Mathematics, Faculty of Civil Engineering, Ss. Cyril and Methodius University, Skopje,
Partizanski Odredi
24, P.O. box 560, 1000 Skopje, Macedonia}
\email{velinovd@gf.ukim.edu.mk}

\author{Marko Kosti\' c}
\address{Faculty of Technical Sciences,
University of Novi Sad,
Trg D. Obradovi\' ca 6, 21125 Novi Sad, Serbia}
\email{marco.s@verat.net}

\author{Stevan Pilipovi\' c}
\address{Department for Mathematics and Informatics,
University of Novi Sad,
Trg D. Obradovi\' ca 4, 21000 Novi Sad, Serbia}
\email{pilipovic@dmi.uns.ac.rs}

{\renewcommand{\thefootnote}{} \footnote{
\\ \text{  }  \ \    {\it Key words and phrases.} Degenerate $C$-distribution cosine functions, degenerate $C$-ultradistribution cosine functions, degenerate $C$-distribution semigroups, degenerate $C$-ultradistribution semigroups,
degenerate integrated $C$-semigroups, degenerate integrated $C$-cosine functions, subgenerators, multivalued linear operators, locally convex spaces.
\\  \text{  }  \ \ This research is partially supported by grant 174024 of Ministry
of Science and Technological Development, Republic of Serbia.}}

\begin{abstract}
The main purpose of this paper is to investigate degenerate $C$-(ultra)distribution cosine functions in the setting of barreled
sequentially complete locally convex spaces.
In our approach, the infinitesimal generator of a degenerate $C$-(ultra)distribution cosine function is a multivalued linear operator
and the regularizing operator $C$ is not necessarily injective. We provide a few important theoretical novelties, considering also exponential subclasses of degenerate $C$-(ultra)distribution cosine functions.
\end{abstract}
\maketitle

\section{Introduction and Preliminaries}

We have recently analyzed in \cite{polugrupe-neijektivne-cds} and \cite{polugrupe-neijektivne-cds-prim}, the classes of degenerate $C$-distribution semigroups and degenerate $C$-ultradistribution semigroups in the setting of barreled
sequentially complete locally convex spaces. We refer to \cite{carol}, \cite{faviniyagi}, \cite{FKP},
\cite{me152} and \cite{svir-fedorov} for further information about well-posedness of abstract degenerate differential equations of first order.
In this way we continue the researches raised in \cite{kolje}, \cite{kolje1} and \cite{pm} (see also \cite{baskakov-chern}, \cite{ki90}, \cite{ku112}, \cite{isna-maiz} and
\cite{me152}-\cite{me155}).  The operator $C$ is not injective, in general.
The analysis of $C$-ultradistribution cosine functions is new even in non-degenerate case, with $C=I$ and the pivot space being one of Banach's, while the analysis of $C$-distribution cosine functions is new in locally convex spaces.

The organization of paper can be briefly described as follows.
Section 2 and Section 3 are devoted to degenerate $C$-distribution cosine functions and degenerate $C$-ultradistribution cosine functions as well as to connection of the degenerate $C$-distribution cosine functions and degenerate integrated $C$-cosine functions. Our theory is illustrated by the examples given in Section 4. In the Appendix are recollected the basic facts about fractionally integrated $C$-semigroups and fractionally integrated $C$-cosine functions in locally convex spaces.\\
We use the standard notation throughout the paper;
$E$ is a Hausdorff sequentially complete
locally convex space over the field of complex numbers, SCLCS for short.  For the sake of brevity
and better exposition, our standing assumption henceforth will be that
the state space $ E$ is barreled.
The exponential region $E(a,b)$ has been defined for the first time by W. Arendt, O. El--Mennaoui and V. Keyantuo in \cite{a22}:
$$
E(a,b):=\Bigl\{\lambda\in\mathbb{C}:\Re\lambda\geq b,\:|\Im\lambda|\leq e^{a\Re\lambda}\Bigr\} \ \ (a,\ b>0).
$$


The Schwartz spaces of test functions $\mathcal{D}=C_0^{\infty}(\mathbb{R})$, $\mathcal{E}=C^{\infty}(\mathbb{R})$ and ${\mathcal S}(\mathbb R)$
carry the usual topologies.
If $\Omega$ is a non-empty open set in ${\mathbb R},$ then the symbol $\mathcal{D}_{\Omega}$ denotes the subspace of $\mathcal{D}$ consisting of those functions $\varphi \in \mathcal{D}$ for which supp$(\varphi) \subseteq \Omega;$ $\mathcal{D}_{0}\equiv \mathcal{D}_{[0,\infty)}.$
The spaces
$\mathcal{D}'(E):=L(\mathcal{D},E)$,
$\mathcal{E}'(E):=L(\mathcal{E},E)$ and
$\mathcal{S}'(E):=L(\mathcal{S},E)$
are topologized in the usual way; the symbols
$\mathcal{D}'_{\Omega}(E)$,
$\mathcal{E}'_{\Omega}(E)$ and $\mathcal{S}'_{\Omega}(E)$ denote the subspaces of
$\mathcal{D}'(E)$, $\mathcal{E}'(E)$ and $\mathcal{S}'(E)$,
respectively, containing $E$-valued distributions
whose supports are contained in $\Omega ;$ $\mathcal{D}'_{0}(E)\equiv \mathcal{D}'_{[0,\infty)}(E)$, $\mathcal{E}'_{0}(E)\equiv \mathcal{E}'_{[0,\infty)}(E)$, $\mathcal{S}'_{0}(E)\equiv \mathcal{S}'_{[0,\infty)}(E).$
If $E={\mathbb C},$ then the above spaces are  the classical ones.
By a regularizing
sequence in $\mathcal{D}$ we mean any sequence $(\rho_n)_{n\in {\mathbb N}}$ in
$\mathcal{D}_0$ for which there exists a function $\rho\in\mathcal{D}$ satisfying $\int_{-\infty}^{\infty}\rho
(t)\,dt=1,$ supp$(\rho)\subseteq [0,1]$ and $\rho_n(t)=n\rho(nt)$,
$t\in\mathbb{R},$ $n\in {\mathbb N}.$
Let $\varphi,\psi\in L^1(0,\infty)$. Then the convolution products $\varphi*\psi$
and $\varphi*_0\psi$ are defined by
$$
\varphi*\psi(t):=\int\limits_{-\infty}^{\infty}\varphi(t-s)\psi(s)\,ds\mbox{ and }
\varphi*_0
\psi(t):=\int\limits^t_0\varphi(t-s)\psi(s)\,ds,\;t\;\in\mathbb{R}.
$$
Notice that $\varphi*\psi=\varphi*_0\psi$,if they are supported by $[0,\infty).$
Given $\varphi\in\mathcal{D}$ and $f\in\mathcal{D}'$, or $\varphi\in\mathcal{E}$ and $f\in\mathcal{E}'$,
we define the convolution $f*\varphi$ by $(f*\varphi)(t):=f(\varphi(t-\cdot))$, $t\in\mathbb{R}$.
For $f\in\mathcal{D}'$, or for $f\in\mathcal{E}'$,
define $\check{f}$ by $\check{f}(\varphi):=f(\varphi (-\cdot))$, $\varphi\in\mathcal{D}$ ($\varphi\in\mathcal{E}$).

Let $G$ be an $E$-valued distribution and let $f\in L^1_{\mbox{loc}}({\mathbb R},E)$.
Then $G^{(n)}$ ($n\in {\mathbb N}$) and $ hG$ ($h\in {\mathcal E}$); the regular $E$-valued distribution ${\mathbf f}$ is defined by ${\mathbf f}(\varphi):=\int_{-\infty}^{\infty}\varphi (t)  f(t) \, dt$ ($\varphi \in {\mathcal D}$).
The following lemma can be deduced as in the scalar-valued case.

\begin{lem}\label{polinomi}
Suppose that $0<\tau \leq \infty,$ $n\in {\mathbb N}$. If $f : (0,\tau) \rightarrow E$ is a continuous function and
$
\int^{\tau}_{0}\varphi^{(n)}(t)f(t)\, dt=0,\quad \varphi \in {\mathcal D}_{(0,\tau)}.
$
Then there exist elements $x_{0},\cdot \cdot \cdot, x_{n-1}$ in $E$ such that $f(t)=\sum^{n-1}_{j=0}t^{j}x_{j},$ $t\in (0,\tau).$
\end{lem}

Let $\tau>0,$ and let $X$ be a general Hausdorff locally convex space (not necessarily sequentially complete). Following \cite{sch16}, $G\in {\mathcal D}'(X)$ is of finite order on the interval $(-\tau,\tau)$ iff
there exist an integer $n\in {\mathbb N}_{0}$ and an $X$-valued continuous function $f : [-\tau,\tau] \rightarrow X$
such that
$
G(\varphi)=(-1)^{n}\int^{\tau}_{-\tau}\varphi^{(n)}(t)f(t)\, dt,\quad \varphi \in {\mathcal D}_{(-\tau,\tau)};
$
$G$ is of finite order iff $G$ is of finite order on any finite interval $(-\tau,\tau).$
In the case that $X$ is a quasi-complete (DF)-space, then each $X$-valued distribution is of finite order.

Henceforth we assume that $(M_p)$ is a sequence of positive real numbers
such that $M_0=1$ and the following conditions hold:\vspace{0.1cm} \newline \noindent
(M.1):
$
M_p^2\leq M_{p+1} M_{p-1},\;\;p\in\mathbb{N};
$\\
(M.2):
$
M_p\leq AH^p\sup_{0\leq i\leq p}M_iM_{p-i},\;\;p\in\mathbb{N},\mbox{ for some }A,\ H>1;$\\
(M.3)':
$
\sum_{p=1}^{\infty}\frac{M_{p-1}}{M_p}<\infty.
$\\
Every employment of the condition\vspace{0.1cm} \newline \noindent
(M.3):
$
\sup_{p\in\mathbb{N}}\sum_{q=p+1}^{\infty}\frac{M_{q-1}M_{p+1}}{pM_pM_q}<\infty,
$\\
which is stronger than (M.3)', will be explicitly emphasized.

The associated
function of sequence $(M_p)$ is defined by
$M(\rho):=\sup_{p\in\mathbb{N}}\ln\frac{\rho^p}{M_p}$,
$\rho>0$; $M(0):=0,$ $M(\lambda):=M(|\lambda|),$
$\lambda\in\mathbb{C} \setminus [0,\infty).$

The spaces of Beurling,
respectively, Roumieu ultradifferentiable functions are
defined by $\mathcal{D}^{(M_p)}:=\mathcal{D}^{(M_p)}(\mathbb{R})
:=\text{indlim}_{K\Subset\Subset\mathbb{R}}\mathcal{D}^{(M_p)}_K$,
respectively,
$\mathcal{D}^{\{M_p\}}:=\mathcal{D}^{\{M_p\}}(\mathbb{R})
:=\text{indlim}_{K\Subset\Subset\mathbb{R}}\mathcal{D}^{\{M_p\}}_K$, (where $K$ goes through all compact sets in ${\mathbb R}$
where
$\mathcal{D}^{(M_p)}_K:=\text{projlim}_{h\to\infty}\mathcal{D}^{M_p,h}_K$,
respectively, $\mathcal{D}^{\{M_p\}}_K:=\text{indlim}_{h\to 0}\mathcal{D}^{M_p,h}_K$,
\begin{align*}
\mathcal{D}^{M_p,h}_K:=\bigl\{\phi\in C^{\infty}(\mathbb{R}): \text{supp}(\phi) \subseteq K,\;\|\phi\|_{M_p,h,K}<\infty\bigr\},
\end{align*}

\begin{align*}
\|\phi\|_{M_p,h,K}:=\sup\Biggl\{\frac{h^p\bigl|\phi^{(p)}(t)\bigr|}{M_p} : t\in K,\;p\in\mathbb{N}_0\Biggr\}.
\end{align*}
Henceforth the asterisk $*$ stands for both cases.

The spaces of tempered ultradistributions of the Beurling,
resp. the Roumieu type, are defined in \cite{pilip} (cf. also \cite{carmic}) as duals of the corresponding test spaces
$
\mathcal{S}^{(M_p)}(\mathbb{R}):=\text{projlim}_{h\to\infty}\mathcal{S}^{M_p,h}(\mathbb{R}),
\mbox{ resp. }\mathcal{S}^{\{M_p\}}(\mathbb{R}):=\text{indlim}_{h\to 0}\mathcal{S}^{M_p,h}(\mathbb{R}),
$
where
$
\mathcal{S}^{M_p,h}(\mathbb{R}):=\bigl\{\phi\in C^\infty(\mathbb{R}):\|\phi\|_{M_p,h}<\infty\bigr\},\;h>0,$

$$\|\phi\|_{M_p,h}:=\sup\Biggl\{\frac{h^{\alpha+\beta}}{M_\alpha M_\beta}(1+t^2)^{\beta/2}|\phi^{(\alpha)}(t)|:t\in\mathbb{R},
\;\alpha,\;\beta\in\mathbb{N}_0\Biggr\}.
$$

Let $\emptyset \neq \Omega \subseteq {\mathbb R}.$
As in the case of distributions, put
$\mathcal{D}'^*(E):=L(\mathcal{D}^*, E)$, $\mathcal{S}'^*(E):=L(\mathcal{S}^{*}, E),$ $\mathcal{D}^{*}_{\Omega}$, $\mathcal{D}^{\ast}_0$, $\mathcal{E}'^{*}_{\Omega}$, $\mathcal{E}'^{*}_{0}$, $\mathcal{D}'^{*}_{\Omega}(E)$,  $\mathcal{D}'^{*}_{0}(E)$.
The multiplication by a function $a\in {\mathcal E}^{\ast}(\Omega)$,
convolution of scalar valued ultradistributions (ultradifferentiable
functions), and the notion of a regularizing sequence in $\mathcal{D}^*,$ are defined as in the case of
distributions.\\
\indent

Let $\eta\in\mathcal{D}_{[-2,-1]}$  ($\eta\in \mathcal{D}^*_{[-2,-1]}$) be a fixed test function satisfying $\int_{-\infty}^{\infty}\eta (t)\,dt=1$.
Then, for every fixed $\varphi\in\mathcal{D}$ ($\varphi\in\mathcal{D}^*$), we define the antiderivative $I(\varphi)$:
$$
I(\varphi)(x):=\int\limits_{-\infty}^x
\Biggl[\varphi(t)-\eta(t)\int\limits_{-\infty}^{\infty}\varphi(u)\,du\Biggr]\,dt,
\;\;x\in\mathbb{R}.
$$
For every $\varphi\in\mathcal{D}$ ($\varphi\in\mathcal{D}^*$) and $n\in {\mathbb N},$
$I(\varphi)\in\mathcal{D}$ ($I(\varphi) \in\mathcal{D}^*$), $I^{n}(\varphi^{(n)})=\varphi ,$
$\frac{d}{dx}I(\varphi)(x)=\varphi(x)-\eta(x)\int_{-\infty}^{\infty}\varphi(u)\,du$, $x\in\mathbb{R}$ as well as that, for every $\varphi\in\mathcal{D}_{[a,b]}$  ($\varphi\in\mathcal{D}_{[a,b]}^{\ast}$), where $-\infty<a<b<\infty$, we have:
$
\text{supp}( I(\varphi))\subseteq[\min(-2,a),\max(-1,b)].
$
This simply implies that, for every $\tau>2,$ $-1<b<\tau$  and for every $m,\ n\in {\mathbb N}$ with $m\leq n,$ we have: $I^0(\varphi):=\varphi$, $\varphi\in{\mathcal D}$ and
\begin{align}\label{jednazba-zaat}
I^{n}\bigl(\mathcal{D}_{(-\tau,b]}\bigr)\subseteq \mathcal{D}_{(-\tau,b]}\mbox{ and }
\frac{d^{m}}{dx^{m}}I^{n}(\varphi)(x)=I^{m-n}\varphi(x),\quad \varphi \in {\mathcal D} \ (\varphi\in\mathcal{D}^*) ,\ x\geq 0.
\end{align}
Define now $G^{-1}$ by
\begin{equation}\label{broj}
G^{-1}(\varphi):=-G(I(\varphi)),\quad \varphi\in\mathcal{D}\,\,\, (\varphi\in\mathcal{D}^*).\end{equation}
It is well known that
$G^{-1}\in\mathcal{D}'(L(E))$ and $(G^{-1})'=G$;
more precisely, $-G^{-1}(\varphi')=G(I(\varphi'))=G(\varphi)$, $\varphi\in\mathcal{D}$. The convergence $\varphi_{n}\rightarrow \varphi ,$ $n\rightarrow \infty$
in $\mathcal{D}_{K}^{M_{p},h}$ implies the convergence $I(\varphi_{n}) \rightarrow I(\varphi) ,$ $n\rightarrow \infty$
in $\mathcal{D}_{K'}^{M_{p},h},$ where  $K'=[\min(-2,\inf(K)), \max(-1, \sup(K))]$, the same holds in ultradistributional case. In both cases, $\text{supp}(G)\subseteq[0,\infty)\Rightarrow \text{supp} (G^{-1})\subseteq[0,\infty)$.

\section[The basic properties of degenerate $C$-distribution...]{The basic properties of degenerate $C$-distribution cosine functions and
degenerate $C$-ultradistribution cosine functions
in locally convex spaces}\label{maxx-duo}

Throughout this section, we assume that $E$ is a barreled SCLCS and that $C\in L(E)$ is not necessarily injective operator. We introduce the notions of pre$-(C-DCF)$ and $(C-DCF)$ (pre$-(C-UDCF)$ of $\ast$-class and $(C-UDCF)$ of $\ast$-class) as follows:

\begin{defn}\label{4cd}
An element ${\mathbf G}\in {\mathcal
D}_{0}^{\prime }(L(E))$ (${\mathbf G}\in {\mathcal
D}_{0}^{\prime \ast}(L(E))$) is called a pre$-(C-DCF)$ (pre$-(C-UDCF)$ of $\ast$-class) iff ${\mathbf
G}(\varphi)C=C{\mathbf G}(\varphi),$ $\varphi \in {\mathcal D}$ ($\varphi \in {\mathcal D}^\ast$) and
\[
(CCF_{1}) : {\mathbf G}^{-1}(\varphi \ast _{0}\psi )C={\mathbf
G}^{-1}(\varphi ){\mathbf G}(\psi )+{\mathbf G}(\varphi ){\mathbf
G}^{-1}(\psi ),\quad \varphi ,\ \psi \in {\mathcal D} \ \ (\varphi ,\ \psi \in {\mathcal D}^\ast);
\]
if, additionally,
\[
(CCF_{2}):\qquad x=y=0\;\mbox{iff}\;\text{\ }{\mathbf G}(\varphi
)x+{\mathbf G}^{-1}(\varphi )y=0,\quad  \varphi \in {\mathcal D}_{0} \ \ (\varphi \in {\mathcal D}^{\ast}_{0}),
\]
then ${\mathbf G}$ is called a $C$-distribution cosine function ($C$-ultradistribution cosine function  of $\ast$-class), in\index{$C$-distribution cosine functions}
short $(C-DCF)$  ($(C-UDCF)$  of $\ast$-class). A pre$-(C-DCF)$ (pre-$(C-UDCF)$  of $\ast$-class) ${\mathbf G}$ is called dense iff\index{$C$-distribution cosine functions!dense}
the set ${\mathcal R}({\mathbf G}):=\bigcup_{\varphi \in {\mathcal
D}_{0}}R({\mathbf G}(\varphi))$ (${\mathcal R}({\mathbf G}):=\bigcup_{\varphi \in {\mathcal
D}_{0}^{\ast}}R({\mathbf G}(\varphi))$) is dense in $E.$
\end{defn}

It is clear that $(CCF_{2})$ implies ${\mathcal N}({\mathbf G}):=\bigcap_{\varphi \in {\mathcal D}%
_{0}}N({\mathbf G}(\varphi ))=\{0\}$ and \\
$\bigcap_{\varphi \in {\mathcal D}%
_{0}}N({\mathbf G}^{-1}(\varphi ))=\{0\},$ and that the
assumption ${\mathbf G}\in {\mathcal D}_{0}^{\prime }(L(E))$ implies
${\mathbf G}(\varphi )=0,$ $\varphi \in {\mathcal D} _{(-\infty , 0]}.$
For $\psi \in {\mathcal D},$ we set $\psi_{+}(t):=\psi(t)H(t),$ $t\in {\mathbb R},$
where $H(t)$ denotes the Heaviside function.
Then
$\psi_{+}\in {\mathcal E}^{\prime}_{0},$ $\psi \in {\mathcal D}$
and $\varphi \ast \psi_{+}\in {\mathcal D}_{0}$ for any $\varphi \in {\mathcal D}_{0}.$ The above holds in ultradistributional case, as well.

The following proposition is essential.

\begin{prop}\label{5cp-deg}
Let ${\mathbf G}\in {\mathcal D}_{0}^{\prime }(L(E))$ (${\mathbf G}\in {\mathcal
D}_{0}^{\prime \ast}(L(E))$) and
${\mathbf G}(\cdot)C=C{\mathbf G}(\cdot).$ Then ${\mathbf G}$ is a pre-(C-DCF) in $E$ (pre-(C-UDCF)  of $\ast$-class  in $E$) iff $${\mathcal
G}\equiv \left(
\begin{array}{cc}
{\mathbf G} & {\mathbf G}^{-1} \\
{\mathbf G}^{\prime }-\delta \otimes C & {\mathbf G}
\end{array}
\right) $$ is a pre-(${\mathcal C}$-DS) in $E\oplus E$ (pre-(${\mathcal C}$-UDS)  of $\ast$-class in $E\oplus E$), where $${\mathcal
C}\equiv \left(
\begin{array}{cc}
C & 0 \\
0 & C
\end{array}
\right).$$ Moreover, ${\mathcal G}$ is a (${\mathcal C}$-DS) ((${\mathcal C}$-UDS)  of $\ast$-class) iff
${\mathbf G}$ is a pre-(C-DCF) (pre-(C-UDCF)  of $\ast$-class) which satisfies $(CCF_{2}).$
\end{prop}
\begin{proof}
By a simple calculation we have that ${\mathcal G}$ satisfies ${\mathcal G}(\varphi{\ast}_0\psi)C={\mathcal G}(\varphi){\mathcal G}(\psi)$, for $\varphi,\psi\in{\mathcal D}$, iff the following holds
\begin{itemize}
\item[i)] ${\mathbf G}^{-1}(\varphi{\ast}_0\psi)C={\mathbf G}^{-1}(\varphi){\mathbf G}(\psi)+{\mathbf G}(\varphi){\mathbf G}^{-1}(\psi);$
\item[ii)] ${\mathbf G}(\varphi{\ast}_0\psi)C={\mathbf G}(\varphi){\mathbf G}(\psi)+{\mathbf G}^{-1}(\varphi)({\mathbf G}'-\delta\otimes C)(\psi);$
\item[iii)] ${\mathbf G}'(\varphi{\ast}_0\psi)C=({\mathbf G}'-\delta\otimes C){\mathbf G}(\psi)+{\mathbf G}(\varphi)({\mathbf G}'-\delta\otimes C)(\psi),$  for $\varphi, \psi \in {\mathcal D}$.
\end{itemize}
It will be proven here that $i)\Rightarrow ii)\Rightarrow iii)$. Let i) holds. By
$(\varphi{\ast}_0\psi)'={\varphi}'{\ast}_0{\psi}+\varphi(0)\psi={\varphi}{\ast}_0{\psi}'+\psi(0)\varphi$, for $\varphi,\psi\in{\mathcal D}$, we have
$${\mathbf G}({\varphi}{\ast}_0{\psi})C=-{\mathbf G}^{-1}((\varphi{\ast}_0\psi)')C=-{\mathbf G}^{-1}(\varphi{\ast}_0{\psi}'+\psi(0)\varphi)C=$$
$$=-({\mathbf G}^{-1}(\varphi){\mathbf G}({\psi}')+{\mathbf G}(\varphi){\mathbf G}^{-1}({\psi}')+\delta(\psi)C{\mathbf G}^{-1}(\varphi))=$$
$$={\mathbf G}^{-1}(\varphi){\mathbf G}'(\psi)+{\mathbf G}(\varphi){\mathbf G}(\psi)-\delta(\psi)C{\mathbf G}^{-1}(\varphi),\quad \mbox{for}\,\, \varphi,\psi\in{\mathcal D}.$$
Now, let ii) holds. Then
$${\mathbf G}'({\varphi}{\ast}_0{\psi})C=-{\mathbf G}(({\varphi}{\ast}_0{\psi})')C=-{\mathbf G}({\varphi}'{\ast}_0{\psi}+\varphi(0){\psi})C=$$ $$=-({\mathbf G}({\varphi}'){\mathbf G}(\psi)+{\mathbf G}^{-1}({\varphi}')({\mathbf G}'-\delta\otimes C)(\psi)+\delta(\varphi){\mathbf G}(\psi)C)=$$ $$=({\mathbf G}'-\delta\otimes C)(\varphi){\mathbf G}(\psi)+{\mathbf G}(\varphi)({\mathbf G}'-\delta\otimes C)(\psi),\quad \varphi,\psi\in{\mathcal D},$$ so we obtain iii).
If ${\mathbf G}$ satisfies $(CCF_2)$ then ${\mathcal G}$ satisfies the non-degeneracy condition (C.S.2) (see \cite{polugrupe-neijektivne-cds}). Let ${\mathcal G}$ satisfies (C.S.2). We will prove that ${\mathbf G}$ satisfies $(CCF_2)$. We assume that $x,y\in E$ and ${\mathbf G}(\varphi)x+{\mathbf G}^{-1}(\varphi)y=0$, $\varphi\in{\mathcal D}_0$. One gets that
$$({\mathbf G}'-\delta)(\varphi)x+{\mathbf G}(\varphi)y=-{\mathbf G}({\varphi}')x-\varphi(0)x=-{\mathbf G}^{-1}({\varphi}')y=0,\quad \varphi\in{\mathcal D}_0.$$ Since ${\mathcal G}$ satisfies (C.S.2), then $x=y=0$, so ${\mathbf G}$ satisfies $(CCF_2)$. The proof for ultradistribution case can be given analogously.
\end{proof}

We can prove the following generalization of \cite[Proposition 3.2.4(ii)]{knjigaho}.

\begin{prop}\label{5cp-deg-deg}
Let ${\mathbf G}\in {\mathcal D}_{0}^{\prime }(L(E))$ (${\mathbf G}\in {\mathcal D}_{0}^{\prime \ast}(L(E))$) and
${\mathbf G}(\cdot)C=C{\mathbf G}(\cdot).$ Then the following holds:
\begin{itemize}
\item[(i)] If
${\mathbf G}$ is a pre-(C-DCF)
(pre-(C-UDCF)  of $\ast$-class), then
\begin{align}\label{kisin-cdf}
{\mathbf G}^{-1}\bigl(\varphi \ast \psi _{+}\bigr)C={\mathbf G}^{-1}(\varphi
){\mathbf G}(\psi )+{\mathbf G}(\varphi ){\mathbf G}^{-1}(\psi
),\text{ }\varphi \in {\mathcal D}_{0},\text{ }\psi \in {\mathcal D}.
\end{align}
\item[(ii)] If $(CCF_{2})$ and (\ref{kisin-cdf}) hold, then ${\mathbf G}$ is a (C-DCF)
((C-UDCF) of $\ast$-class).
\end{itemize}
\end{prop}
\begin{proof}

(i) Let ${\mathbf G}$ be a pre-(C-DCF). Then ${\mathcal G}$ is a (C-DS). Then ${\mathcal G}$ is a (C-DS) in $E\oplus E$ and ${\mathcal G}({\psi}_+)={\mathcal G}(\psi)$, for $\psi\in{\mathcal D}$. Hence,
$$\left(\begin{array}{cc}
{\mathbf G}(\varphi\ast{\psi}_+) & {\mathbf G}^{-1}(\varphi\ast{\psi}_+) \\
({\mathbf G}^{\prime }-\delta \otimes C)(\varphi\ast{\psi}_+) & {\mathbf G}(\varphi\ast{\psi}_+)
\end{array}\right) \left(\begin{array}{cc}
x \\
y
\end{array}\right)=$$
$$=\left(\begin{array}{cc}
{\mathbf G}(\varphi) & {\mathbf G}^{-1}(\varphi) \\
({\mathbf G}^{\prime }-\delta \otimes C)(\varphi) & {\mathbf G}(\varphi)
\end{array}\right)\left(\begin{array}{cc}
{\mathbf G}({\psi}) & {\mathbf G}^{-1}({\psi}) \\
({\mathbf G}^{\prime }-\delta \otimes C)({\psi}) & {\mathbf G}({\psi})
\end{array}\right)\left(\begin{array}{cc}
x \\
y
\end{array}\right)$$ for every $\varphi\in{\mathcal D}_0$, $\psi\in{\mathcal D}$, $x,y\in E$. If we choose $x=0$, then we obtain ${\mathbf G}^{-1}({\varphi}{\ast}{\psi}_0)C={\mathbf G}^{-1}(\varphi){\mathbf G}(\psi)+{\mathbf G}(\varphi){\mathbf G}^{-1}(\psi)$, $\varphi\in{\mathcal D}_0$, $\psi\in{\mathcal D}$.\\
(ii) Let now (\ref{kisin-cdf}) and $(CCF_2)$ are fulfilled. Then ${\mathcal G}$ is satisfying non-degeneracy condition (C.S.2) (see \cite{polugrupe-neijektivne-cds}). By (\ref{kisin-cdf}), we have
$${\mathbf G}^{-1}(\varphi{\ast}\psi)C={\mathbf G}^{-1}(\varphi){\mathbf G}(\psi)+{\mathbf G}(\varphi){\mathbf G}^{-1}(\psi),\quad \varphi,\psi\in{\mathcal D}_0$$ and
consequently
$${\mathbf G}(\varphi{\ast}\psi)C={\mathbf G}(\varphi){\mathbf G}(\psi)-{\mathbf G}^{-1}(\varphi){\mathbf G}^{-1}(\varphi)({\mathbf G}'-\delta\otimes C)(\psi)$$
$$({\mathbf G}'-\delta\otimes C)(\varphi\ast\psi)=({\mathbf G}'-\delta\otimes C)(\varphi){\mathbf G}(\psi)+{\mathbf G}(\varphi)({\mathbf G}'-\delta\otimes C)(\psi),\quad \varphi,\psi\in{\mathcal D}_0.$$
We get the ${\mathcal G}$ is a pre-(C-DS). Now for $\varphi\in{\mathcal D}_0$ and $\psi\in{\mathcal D}$ we obtain
$${\mathbf G}({\varphi}{\ast}{\psi}_+)C=-{\mathbf G}^{-1}((\varphi{\ast}_0{\psi}_+)')C=-{\mathbf G}^{-1}({\varphi}'{\ast}_0{\psi}_+{\varphi}(0){\psi}_+)C=$$
$$=-({\mathbf G}^{-1}({\varphi}'){\mathbf G}(\psi)+{\mathbf G}({\varphi}'){\mathbf G}^{-1}(\psi))={\mathbf G}(\varphi){\mathbf G}(\psi)+{\mathbf G}'(\varphi){\mathbf G}^{-1}(\psi).$$
Since $(\varphi{\ast_0}{\psi}_+)'=({\varphi}{\ast}_0{\psi}'_+)+\psi(0)\varphi$, $\varphi\in{\mathcal D}_0$, $\psi\in{\mathcal D}$, we get
$${\mathbf G}({\varphi}{\ast}{\psi}_+)C=-{\mathbf G}^{-1}(({\varphi}{\ast}_0{\psi}_+)')C=-{\mathbf G}^{-1}({\varphi}{\ast}_0(\psi)'_+ +\psi(0)\varphi)C=$$
$$=-({\mathbf G}^{-1}(\varphi){\mathbf G}({\psi}')+{\mathbf G}(\varphi){\mathbf G}^{-1}({\psi}'))-\psi(0){\mathbf G}^{-1}(\varphi)C=$$
$$={\mathbf G}(\varphi){\mathbf G}(\psi)+{\mathbf G}^{-1}(\varphi)({\mathbf G}'-\delta\otimes C)(\psi)$$
and
$$({\mathbf G}'-\delta\otimes C)(\varphi\ast{\psi}_+)={\mathbf G}'(\varphi\ast{\psi}_+)C=-{\mathbf G}(({\varphi}{\ast}_0{\psi}_+)')C=$$

$$=-{\mathbf G}({\varphi}'{\ast}_0{\psi}_+)C=-({\mathbf G}({\varphi}'){\mathbf G}(\psi)+{\mathbf G}^{-1}({\varphi}')({\mathbf G}'-\delta\otimes C)(\psi))=$$
$$=({\mathbf G}'-\delta\otimes C)(\varphi){\mathbf G}(\psi)+{\mathbf G}(\varphi)({\mathbf G}'-\delta\otimes C)(\psi).$$
Hence, we obtained that ${\mathcal G}$ is a (C-DS) and by Proposition \ref{5cp-deg} we get that ${\mathbf G}$ is a (C-DCF). The proof is analogous for the ultradistribution case.
\end{proof}

If
${\mathbf G}$ is a pre-(C-DCF)
(pre-(C-UDCF)  of $\ast$-class), then we can almost directly prove that
the dual $\mathbf{G}(\cdot)^*$ is a pre-($C^*$-DCF) (pre-($C^*$-UDCF) of $\ast$-class) on $E^*$ satisfying
$\mathcal{N}(\mathbf {G}^*)=\overline{\mathcal{R}(\mathbf{G})}^{\circ},$ and that the reflexivity of
$E$ additionally implies that $\mathcal{N}(\mathbf{G})=\overline{\mathcal{R}(\mathbf{G}^*)}^{\circ}$.

\begin{prop}\label{kisinski-second}
Suppose that ${\mathbf G}\in {\mathcal D}^{\prime}_{0}(L(E))$ (${\mathbf G}\in {\mathcal D}^{\prime \ast}_{0}(L(E))$) and ${\mathbf G}(\cdot)C=C{\mathbf G}(\cdot).$
Then ${\mathbf G}$ is a pre-(C-DCF) (pre-(C-UDCF) of $\ast$-class) iff for every $\varphi,\ \psi \in {\mathcal D}$ ($\varphi,\ \psi \in {\mathcal D}^{\ast}$), we have:
\begin{align*}
{\mathbf G}^{-1}(\varphi){\mathbf G}^{\prime}(\psi)-{\mathbf G}^{\prime}(\varphi){\mathbf G}^{-1}(\psi)=\psi(0){\mathbf G}^{-1}(\varphi)C-\varphi(0){\mathbf G}^{-1}
(\psi)C.
\end{align*}
\end{prop}
(see (\ref{broj}) for ${\mathbf G}^{-1}(\varphi)$).
\begin{proof}
Having on mind that (see \cite[Proposition 4.5]{polugrupe-neijektivne-cds} and \cite[Proposition 2.5]{polugrupe-neijektivne-cds-prim})
$${\mathcal G}({\varphi}'){\mathcal G}(\psi)-{\mathcal G}(\varphi){\mathcal G}({\psi}')=\psi(0){\mathcal G}({\varphi})C-{\varphi}(0){\mathcal G}(\psi)C,\quad\quad \varphi,\psi\in{\mathcal D}, (\varphi,\psi\in{\mathcal D}^{\ast})$$ and by Proposition \ref{5cp-deg} we obtain the statement of the proposition.
\end{proof}

Assume ${\mathbf G}$ is a pre-$(C-DCF)$ (pre-$(C-UDCF)$ of $\ast$-class). Then we define the (integral) generator
${\mathbf A}$ of ${\mathbf G}$ by
\[
{\mathbf A}:=\Bigl\{(x,y)\in E \oplus E : {\mathbf
G}^{-1}\bigl(\varphi^{\prime \prime}\bigr)x={\mathbf G}^{-1}(\varphi )y \mbox{
for all }\varphi \in {\mathcal D}_{0} \Bigr\}.
\]
Then ${\mathbf A}$ is a closed multi-valued linear operator (MLO) and it can be easily seen that ${\mathbf A} \subseteq C^{-1}{\mathbf A}C,$
with the equality in the case that the operator $C$ is injective. If $(CCF_{2})$ holds, then it is clear that ${\mathbf A}=A$ is a closed
single-valued linear operator.

Furthermore, we can extend the assertion of \cite[Lemma 3.4.7]{knjigah} in our context:

\begin{lem}\label{kisinskija-second}
Let ${\mathbf A}$ be the generator of a pre-(C-DCF)  (pre-(C-UDCF) of $\ast$-class) ${\mathbf G}.$ Then ${\mathcal A}\subseteq {\mathcal B},$ where ${\mathcal A}\equiv \left(
\begin{array}{cc}
0 & I \\
{\mathbf A} & 0
\end{array}
\right) $ and ${\mathcal B}$ is the generator of ${\mathcal G}.$
Furthermore, $(x,y)\in {\mathbf A} \Leftrightarrow
\Bigl(\binom{x}{0},\binom{0}{y}\Bigr)\in {\mathcal B}$ and ${\mathcal B}$ is single-valued iff ${\mathbf G}$ is a (C-DCF)  ((C-UDCF) of $\ast$-class).
\end{lem}
\begin{proof}
Let $(x,y)\in{\mathbf A}$. Then $\Biggl(\binom{x}{0},\binom{0}{y}\Biggr)\in {\mathcal A},\quad x\in E$ and consequently,
$\Biggl(\binom{x}{0},\binom{0}{y}\Biggr)\in {\mathcal B}$. Now, let $\Biggl(\binom{x}{0},\binom{0}{y}\Biggr)\in {\mathcal B}$ and fix $\varphi\in{\mathcal D}_0$. Then ${\mathcal G}(-{\varphi}')\binom{x}{0}={\mathcal G}(\varphi)\binom{0}{y}$ and by the definition of ${\mathcal G}$,
$$\left(\begin{array}{cc}
{\mathbf G}(-{\varphi}') & {\mathbf G}^{-1}(-{\varphi}') \\
({\mathbf G}^{\prime }-\delta \otimes C)(-{\varphi}') & {\mathbf G}(-{\varphi}')
\end{array}\right) \left(\begin{array}{cc}
x \\
0
\end{array}\right)=$$
$$=\left(\begin{array}{cc}
{\mathbf G}(\varphi) & {\mathbf G}^{-1}(\varphi) \\
({\mathbf G}^{\prime }-\delta \otimes C)(\varphi) & {\mathbf G}(\varphi)
\end{array}\right)\left(\begin{array}{cc}
0 \\
y
\end{array}\right)$$
Thereby, ${\mathbf G}(-{\delta}')x={\mathbf G}^{-1}(\varphi)y$, i.e. ${\mathbf G}^{-1}({\varphi}'')x={\mathbf G}(\varphi)y$. This implies that $(x,y)\in{\mathbf A}$.\\
Suppose now that ${\mathbf G}$ is a (${\mathbf C}$-DCF) ((${\mathbf C}$-UDCF)  of $\ast$-class) generated by ${\mathbf A}.$ Then Proposition \ref{5cp-deg} yields that
${\mathcal G}$ is a (C-DS) ((C-UDS)  of $\ast$-class). This implies that the integral generator ${\mathcal B}$ of ${\mathcal G}$ is single-valued and that the operator ${\mathcal C}$ is injective.
\end{proof}

Before proceeding further, it is worth observing that
\begin{align}\label{phase-space-komutiranje}
\Biggl(\binom{0}{x},\binom{x}{0}\Biggr)\in {\mathcal B},\quad x\in E.
\end{align}
We can apply Lemma \ref{kisinskija-second} in order to see that the integral generator ${\mathbf A}$ of ${\mathbf G}$ is single-valued and that the operator $C$ is injective.

Even in the case that $C=I,$ it is not clear whether the assumption that the integral generator ${\mathbf A}$ of a  pre-(${\mathbf C}$-DCF) (pre-(${\mathbf C}$-UDCF)  of $\ast$-class) ${\mathbf G}$ is single-valued implies $(CCF_{2})$ for ${\mathbf G}.$

Let ${\mathbf G}$ be a pre-(C-DCF) (pre-(C-UDCF) of $\ast$-class) generated by ${\mathbf A}.$

\begin{lem}{}
\begin{itemize}
\item[(a)] Let $\psi \in {\mathcal D}$ \ ($%
\psi \in {\mathcal D}^{\ast}$) and $x,\ y\in E$. Then
$({\mathbf G}(\psi )x,y)\in {\mathbf A}$ iff
$$
{\mathbf G}(\psi ^{\prime\prime})x+\psi ^{\prime}(0)Cx-y\in \bigcap_{\varphi \in {\mathcal D}%
_{0}}N\Bigl({\mathbf G}^{-1}(\varphi )\Bigr) \ \ \Biggl(  \in \bigcap_{\varphi \in {\mathcal D}^{\ast}%
_{0}}N\Bigl({\mathbf G}^{-1}(\varphi )\Bigr) \Biggr).
$$
\item[(b)] $({\mathbf G}(\psi )x$, ${\mathbf G}(\psi ^{\prime\prime})x+\psi ^{\prime}(0)Cx)\in {\mathbf A}$, $%
\psi \in {\mathcal D}$ \ ($
\psi \in {\mathcal D}^{\ast}$), $x\in E$.
\item[(c)] $({\mathbf G}^{-1}(\psi )x, -{\mathbf G}(\psi^{\prime})x-\psi (0)Cx)\in {\mathbf A}$, $\psi
\in {\mathcal D}$
\ ($
\psi \in {\mathcal D}^{\ast}$), $x\in E$.
\item[(d)] ${\mathbf G}(\varphi*_0\psi)Cx-{\mathbf G}(\varphi){\mathbf G}(\psi)x\in {\mathbf A}{\mathbf G}^{-1}(\varphi){\mathbf G}^{-1}(\psi)x$, $\varphi,\,\psi\in\mathcal{D}$ ($\varphi,\,\psi\in\mathcal{D}^{\ast}$), $x\in E$.
\end{itemize}
\end{lem}
\begin{proof}
(a) Clearly $({\mathbf G}(\psi)x,y)\in{\mathbf A}$ iff ${\mathbf G}'(\varphi){\mathbf G}(\psi)x={\mathbf G}^{-1}(\varphi)y$, $\varphi\in{\mathcal D}_0$. This is equivalent to ${\mathbf G}'(\varphi{\ast}_0\psi)x-{\mathbf G}(\varphi){\mathbf G}'(\psi)x+\psi(0)C{\mathbf G}(\varphi)x={\mathbf G}^{-1}(\varphi)y.$ By the same arguments used in the proof of Proposition \ref{5cp-deg} we obtain that ${\mathbf G}({\varphi}'')x+{\varphi}'(0)Cx-y\in\bigcap_{\varphi \in {\mathcal D}%
_{0}}N\Bigl({\mathbf G}^{-1}(\varphi )\Bigr)$.\\
(b) This is a consequence of (a).\\
(c) Let we recall that ${\mathbf G}^{-1}(\psi)=-{\mathbf G}(I(\psi))$ and that $\frac{d}{dt}I(\psi)(t)=\psi(t)-\alpha(t)\cdot\int_{-\infty}^{+\infty}\psi(u)\, du$, $t\in{\mathbb R}$. Then $\frac{d^2}{dt^2}I(\psi)(t)={\psi}'(t)-{\alpha}'(t)\cdot\int_{-\infty}^{+\infty}\psi(u)\, du$, $t\in{\mathbb R}$. Since $\alpha\in{\mathcal D}_{[-2,-1]}$ and ${\mathbf G}\in{\mathcal D}'_0(L(E))$, we obtain $(I(\psi))'(0)=\psi(0)$ and ${\mathbf G}((I(\psi))'')={\mathbf G}({\psi}'-{\alpha}'\cdot\int_{-\infty}^{+\infty})={\mathbf G}({\psi}')$. By (a), ${\mathbf A}{\mathbf G}^{-1}({\psi})x=-{\mathbf A}{\mathbf G}(I(\psi))x=-[{\mathbf G}((I(\psi))'')x+(I(\psi))'(0)Cx]=-{\mathbf G}({\psi}')x-{\psi}(0)Cx$.\\
(d) Since ${\mathbf A}$ generates ${\mathbf G}$ and ${\mathbf G}(\varphi)=-{\mathbf G}^{-1}({\varphi}')$, $\varphi\in{\mathcal D}$, we have that $${\mathbf G}({\varphi}{\ast}_0{\psi})Cx=-{\varphi}(0)C{\mathbf G}^{-1}({\psi})x-{\mathbf G}^{-1}({\varphi}'{\ast}_0{\psi})x=$$
$$={\mathbf G}(\varphi){\mathbf G}(\psi)x+(-{\varphi}(0)C-{\mathbf G}({\varphi}')){\mathbf G}^{-1}({\psi})x={\mathbf G}(\varphi){\mathbf G}(\psi)x+{\mathbf A}{\mathbf G}(\varphi){\mathbf G}^{-1}(\psi)x,$$ for $x\in E$. The ultradistribution case can be shown by the same arguments.
\end{proof}

If ${\mathbf G}$ is a (${\mathbf C}$-DCF) ((${\mathbf C}$-UDCF)  of $\ast$-class) generated by ${\mathbf A},$ then the operators ${\mathcal B}$ and ${\mathbf A}$ are single-valued; then a similar line of reasoning as in the proof of \cite[Proposition 3.4.8(iii)-(iv)]{knjigah} shows that, for every $\psi \in {\mathcal D}$ \ ($
\psi \in {\mathcal D}^{\ast}$), we have
${\mathbf G}(\psi ) {\mathbf A}\subseteq {\mathbf A}{\mathbf G}(\psi )$
and
${\mathbf G}^{-1}(\psi ) {\mathbf A}\subseteq {\mathbf A}{\mathbf G}^{-1}(\psi ).$ 

\begin{thm}\label{fundamentalna-second}
Suppose that ${\mathbf G}\in {\mathcal D}^{\prime}_{0}(L(E))$ (${\mathbf G}\in {\mathcal D}^{\prime \ast}_{0}(L(E))$), ${\mathbf G}(\cdot)C=C{\mathbf G}(\cdot),$
and ${\mathbf A}$ is a closed MLO on $E$ satisfying that $\mathbf{G}(\cdot){\mathbf A}\subseteq {\mathbf A}{\mathbf G}(\cdot)$ and
\begin{equation}\label{dkenk-second}
\mathbf{G}\bigl(\varphi^{\prime \prime}\bigr)x+\varphi^{\prime}(0)Cx\in {\mathbf A}\mathbf{G}(\varphi)x,\quad x\in E,\ \varphi \in {\mathcal D} \ \ (\varphi \in {\mathcal D}^{\ast}).
\end{equation}
Then the following holds:
\begin{itemize}
\item[(i)] If ${\mathbf A}=A$ is single-valued, then ${\mathbf G}$ is a pre-(C-DCF) (pre-(C-UDCF) of $\ast$-class).
\item[(ii)] If ${\mathbf G}$ satisfies $(CCF_2)$, $C$ is injective and ${\mathbf A}=A$ is single-valued, then ${\mathbf G}$ is a (C-DCF) ((C-UDCF) of $\ast$-class) generated by $C^{-1}AC.$
\item[(iii)] Consider the distribution case. The condition $(CCF_2)$ automatically holds for ${\mathbf G}$.
\end{itemize}
\end{thm}

\begin{proof}
We will only outline the most important details of proof. It can be simply proved that ${\mathcal G}\in {\mathcal D}^{\prime}_{0}(L(E\oplus E))$ (${\mathcal G}\in {\mathcal D}^{\prime \ast}_{0}(L(E\oplus E))$), ${\mathcal G}(\cdot){\mathcal C}={\mathcal C}{\mathcal G}(\cdot),$
and that ${\mathcal A}$ is a closed MLO in $E \oplus E.$ Furthermore, $\mathcal{G}(\cdot){\mathcal A}\subseteq {\mathcal A}{\mathcal G}(\cdot)$ and
\begin{align*}
\mathcal{G}\bigl(-\varphi^{\prime}\bigr)(x\ y)^{T}-\varphi(0){\mathcal C}(x\ y)^{T}\in {\mathcal A}\mathcal{G}(\varphi)(x\ y)^{T},\quad x,\ y\in E,\ \varphi \in {\mathcal D} \ \ (\varphi \in {\mathcal D}^{\ast}).
\end{align*}
By \cite[Remark 4.14]{polugrupe-neijektivne-cds} and \cite[Remark 2.7]{polugrupe-neijektivne-cds-prim} which ones say that
for ${\mathcal G}\in {\mathcal D}^{\prime}_{0}(L(E))$ (${\mathcal G}\in {\mathcal D}^{\prime\ast}_{0}(L(E))$, ${\mathcal G}(\varphi)C=C{\mathcal G}(\varphi),$ $\varphi \in {\mathcal D}$ ($\varphi \in {\mathcal D}^{\ast}$)
and ${\mathcal A}$ is a closed MLO on $E$ satisfying that $\mathcal{G}(\varphi){\mathcal A}\subseteq {\mathcal A}{\mathcal G}(\varphi),$ $\varphi \in {\mathcal D}$ ($\varphi \in {\mathcal D}^{\ast}$) and
$\mathcal{G}\bigl(-\varphi'\bigr)x-\varphi(0)Cx\in {\mathcal A}\mathcal{G}(\varphi)x,\quad x\in E,\ \varphi \in {\mathcal D}\,\, \varphi \in {\mathcal D}^{\ast})$, we have that ${\mathcal G}$ is a pre-(${\mathcal C}$-DS)  in $E \oplus E$ so that (i) follows immediately from Proposition \ref{5cp-deg}. In order to prove (ii), notice that ${\mathcal G}$ satisfies (C.S.2) and again by \cite[Remark 4.14]{polugrupe-neijektivne-cds} and \cite[Remark 2.7]{polugrupe-neijektivne-cds-prim}, we obtain that ${\mathcal G}$ is a pre-(${\mathcal C}$-DS)  in $E \oplus E$ generated by ${\mathcal C}^{-1}{\mathcal A}{\mathcal C}=\bigl(\begin{smallmatrix}0&I\\ C^{-1}AC &0\end{smallmatrix}\bigr).$ Now the part (ii) simply follows from Proposition \ref{5cp-deg}
and Lemma \ref{kisinskija-second}.
The proof of (iii) can be deduced similarly.
\end{proof}

\begin{rem}\label{prvo-zapazanje}
Concerning the assertion (i), its validity is not true in multivalued case (\cite{polugrupe-neijektivne-cds}-\cite{polugrupe-neijektivne-cds-prim}):
Let $C=I,$ let ${\mathbf A}\equiv E\times E,$ and let ${\mathbf G}\in {\mathcal D}^{\prime}_{0}(L(E))$ (${\mathcal G}\in {\mathcal D}^{\prime \ast}_{0}(L(E))$) be arbitrarily chosen. Then ${\mathbf G}$ commutes with ${\mathbf A}$ and (\ref{dkenk-second}) holds but ${\mathbf G}$ need not satisfy $(CCF_1).$
\end{rem}

\begin{rem}\label{lokalno-konveksni-generatori}
Let ${\mathbf G}\in {\mathcal D}_{0}^{\prime }(L(E))$ (${\mathbf G}\in {\mathcal D}^{\prime \ast}_{0}(L(E))$) and ${\mathbf G}(\cdot)C=C{\mathbf G}(\cdot).$ Suppose that ${\mathcal A}=A$ is single-valued and $C$ is injective. If ${\mathbf G}$ is a (C-DCF) in $E$ ((C-UDCF) of $\ast$-class in $E$), then ${\mathcal B}$ is single-valued and we can proceed as in the proof of \cite[Proposition 3.4.8(iii)]{knjigah} so as to conclude that $\mathbf{G}(\cdot){\mathbf A}\subseteq {\mathbf A}{\mathbf G}(\cdot).$ Combining this
fact with  Proposition \ref{5cp-deg}, \cite[Remark 4.14]{polugrupe-neijektivne-cds} and \cite[Remark 2.7]{polugrupe-neijektivne-cds-prim} and the arguments contained in the proof of Theorem \ref{fundamentalna-second}, we get that ${\mathbf G}$ is a (C-DCF) in $E$ ((C-UDCF) of $\ast$-class in $E$)
generated by ${\mathbf A}$ iff ${\mathcal G}$ is a (${\mathcal C}$-DS) in $E\oplus
E$ ((${\mathcal C}$-UDS) of $\ast$-class in $E\oplus
E$) generated by ${\mathcal A}.$ This is an extension of \cite[Theorem 3.2.8(ii)]{knjigah}. In degenerate case,
the integral generator of ${\mathcal G}$ can strictly contain ${\mathcal A}.$
In order to verify this,  let
$E$ be an arbitrary Banach space,
let $P\in L(E),$ and let $P^2=P$.
Define
${\mathbf G}_{P}(\varphi)x:=\int_0^{\infty}\varphi(t)\,dt\,Px$, $x\in E,\varphi\in\mathcal{D}$.
Then ${\mathbf G}_{P}^{-1}(\varphi)x=\int_0^{\infty}t\varphi(t)\,dt\,Px$, $x\in E$, $\varphi\in\mathcal{D}$,
${\mathbf G}_{P}$ is a pre-(DCF) in $E$, and
\[
\{x,y\}\subseteq N(P)\Leftrightarrow {\mathbf G}_{P}(\varphi)x+{\mathbf G}_{P}^{-1}(\varphi)y=0\mbox{ for all }\varphi\in\mathcal{D}_0;
\]
see \cite[Example 3.4.46]{knjigah}. A straightforward computation shows that the integral generator of ${\mathbf G}_{P}$ is the MLO ${\mathbf A}=E \times N(P).$ Furthermore, $([x\ y]^{T}, [z\ u]^{T}) \in {\mathcal A}$ iff $y=z$ and $u\in N(P),$
while $([x\ y]^{T}, [z\ u]^{T}) \in {\mathcal B}$  iff $y-z\in N(P) $ and $u\in N(P)$. Hence, ${\mathcal B}$ strictly contains ${\mathcal A}$.
\end{rem}

\begin{rem}\label{lokalno-konveksni-generatori-jedinstvenost}
Suppose that ${\mathcal A}=A$ is single-valued and $C$ is injective. Since any  (${\mathcal C}$-DS) in $E\oplus
E$ ((${\mathcal C}$-UDS) of $\ast$-class in $E\oplus
E$) is uniquely determined by its generator,
the conclusion established in Remark \ref{lokalno-konveksni-generatori} shows that there exists at most one (C-DCF) in $E$ ((C-UDCF) of $\ast$-class in $E$) generated by
$A.$ Even in the case that $E$ is a Banach space and $C=I,$ this is no longer true in degenerate case. To see this,  let
$E$ be an arbitrary Banach space,
let $P_{1}\in L(E),$ $P_{1}^{2}=P_{1},$ $P_{2}\in L(E),$ $P_{2}^{2}=P_{2},$ $N(P_{1})=N(P_{2})$
and $P_{1}\neq P_{2};$ cf. the previous remark. Then  pre-(DCF)'s ${\mathbf G}_{P_{1}}$ and ${\mathbf G}_{P_{2}}$ are different but have the same integral generator. We can choose, for example, the matricial operators
\[
P_{1}=\Biggl[\begin{matrix}0& 0\\ 0 & 1\end{matrix}\Biggr]\mbox{ and }P_{2}=\Biggl[\begin{matrix}0& 1\\ 0 & 1\end{matrix}\Biggr].
\]
\end{rem}

We continue by stating the following theorem.

\begin{thm}\label{herbs-second}
Let $a>0$, $b>0$ and $\alpha>0.$ Suppose that ${\mathbf A}$ is a closed \emph{MLO} and, for every $\lambda$ which belongs to the set $
E(a,b)
,$ there exists
an operator $H(\lambda)\in L(E)$ so that $H(\lambda){\mathbf A}\subseteq {\mathbf A} H(\lambda),$ $\lambda \in E(a,b),$ $\lambda H(\lambda)x-Cx \in {\mathbf A}[ H(\lambda)x/\lambda],$ $\lambda \in  E(a,b),$ $x\in E,$
$H(\lambda)C=C H(\lambda),$ $\lambda \in  E(a,b),$
$\lambda H(\lambda)x-Cx=H(\lambda)y/\lambda,$ whenever $\lambda \in  E(a,b)$ and $(x,y)\in {\mathbf A},$ and that the mapping $\lambda \mapsto H(\lambda)$ is strongly analytic on $\Omega_{a,b}$ and strongly continuous on $\Gamma_{a,b},$
where $\Gamma_{a,b}$ denotes the upwards oriented boundary of $E(a,b)$
and $\Omega_{a,b}$ the open region which lies to the right of $\Gamma_{a,b}.$ Let the operator family $\{(1+|\lambda|)^{-\alpha}H(\lambda) : \lambda \in E(a,b)\}\subseteq L(E)$ be equicontinuous.
Set
\begin{align*}
\mathbf{G}(\varphi)x:=(-i)\int_{\Gamma_{a,b}}\hat{\varphi}(\lambda)H(\lambda)x\,d\lambda,
\;\;x\in E,\;\varphi\in\mathcal{D}.
\end{align*}
Then $\mathbf{G}$ is a pre-(C-DCF) generated by an extension of  ${\mathbf A}. $
\end{thm}

\begin{proof}
Set
\begin{align*}
F(\lambda):=\Biggl[\begin{matrix} H(\lambda)&H(\lambda)/\lambda\\ \lambda H(\lambda)-C &H(\lambda)\end{matrix}\Biggr],\quad \lambda \in E(a,b)
\end{align*}
and
\begin{align*}
\mathcal{G}(\varphi)[x \ y]^{T}:=(-i)\int_{\Gamma_{a,b}}\hat{\varphi}(\lambda)F(\lambda)[x \ y]^{T}\,d\lambda,
\;\;x,\ y\in E,\;\varphi\in\mathcal{D}.
\end{align*}
The prescribed assumptions imply that the function $ F(\cdot)$ has the properties necessary for applying \cite[Theorem 4.15]{polugrupe-neijektivne-cds} which one gives that $\mathcal{G}(\varphi)$ is pre-(C-DS) generated by an extension of A  . Furthermore,
supp$(\mathbf{G})\subseteq [0,\infty),$ ${\mathbf G}$ commutes with $C$ and
by the connection between (C-DS)'s and (C-DCF)'s (see \cite[Theorem 3.2.6]{knjigaho}) we have that that
$$
{\mathcal
G}= \Biggl[\begin{matrix}
{\mathbf G} & {\mathbf G}^{-1} \\
{\mathbf G}^{\prime }-\delta \otimes C & {\mathbf G}
\end{matrix}\Biggr].
$$
Due to Proposition \ref{5cp-deg}
and Lemma \ref{kisinskija-second}, we obtain that $\mathbf{G}$ is a pre-(C-DCF) generated by an extension of  ${\mathbf A},$
as claimed.
\end{proof}

\begin{rem}\label{filipa-1-frim}
\begin{itemize}
\item[(i)]
Suppose that $C$ is injective, ${\mathbf A}=A$ is single-valued, $\rho_{C}(A) \subseteq E^{2}(a,b)\equiv \{\lambda^{2} : \lambda \in  E(a,b)\}$ and $H(\lambda)=\lambda (\lambda^{2} -{\mathcal A})^{-1}C,$ $\lambda \in E^{2}(a,b).$ Then ${\mathcal G}$ is a (C-DCF) generated by $C^{-1}AC.$
Even in the case that $C=I,$ the integral generator ${\mathbf A}$ of ${\mathbf G},$ in multivalued case, can strictly contain $C^{-1}{\mathbf A}C.$
\item[(ii)] Let ${\mathbf A}$ be a closed MLO, let $C$ be injective and commute with ${\mathbf A},$ and let $\rho_{C}({\mathbf A}) \subseteq E^{2}(a,b).$ Then the choice $H(\lambda)=\lambda (\lambda^{2} -{\mathcal A})^{-1}C,$ $\lambda \in E(a,b)$
is always possible (\cite{FKP}).
\item[(iii)] In ultradistributional case, it is necessary to replace the exponential region $E(a,b)$ from the formulation
of Theorem \ref{herbs-second} with a corresponding ultra-logarithmic region. Define the operator $\mathbf{G}(\varphi)$ similarly as above.
In non-degenerate case (${\mathbf A}=A$ single-valued, $C$ injective),  it can be proved that $\mathbf{G}(\varphi)$ is a pre-(C-UDCF)
generated by an extension of  ${\mathbf A};$
unfortunately, we do not know then whether
$\mathbf{G}$ has to satisfy $(CCF_1)$ in degenerate case.
\end{itemize}
\end{rem}


The analysis of degenerate almost $C$-(ultra)distribution cosine functions is without the scope of this paper. For more details, see \cite{pm} and \cite[Subsection 3.4.5]{knjigah} and \cite[pp. 380-384]{knjigaho}.

\section[Relations between degenerate $C$-distribution cosine functions...]{Relations between degenerate $C$-distribution cosine functions and degenerate integrated $C$-cosine functions}\label{maxx-frim-primw}

We start this section by stating the following fundamental result.

\begin{thm}\label{lokal-int-C-secondinjo}
Let $\mathbf{G}$ be a pre-(C-DCF) generated by ${\mathbf A}$, and let $\mathbf{G}$ be of finite order.
Then, for every $\tau>0$, there exist a number $n_{\tau}\in\mathbb{N}$
and a local $n_{\tau}$-times integrated $C$-cosine function $(C_{n_{\tau}}(t))_{t\in [0,\tau)}$  such that
\begin{align}\label{utf-88-second}
{\mathcal G}(\varphi)=(-1)^{n_{\tau}}\int \limits^{\infty}_{0}\varphi^{(n_{\tau})}(t)C_{n_{\tau}}(t)\, dt,\quad \varphi \in {\mathcal D}_{(-\tau ,\tau)}.
\end{align}
Furthermore, $(C_{n_{\tau}}(t))_{t\in [0,\tau)}$ is an $n_{\tau}$-times integrated $C$-cosine existence family with a subgenerator ${\mathbf A}.$
\end{thm}

\begin{proof}
Let $
{\mathcal G}$ and $
{\mathcal C}$ be as in the formulation of
Proposition \ref{5cp-deg}, and let ${\mathcal A}$ be the MLO defined in Lemma \ref{kisinskija-second}. Then ${\mathcal G}$ is a pre-(${\mathcal C}$-DS) in $E \oplus E$ generated by a closed MLO ${\mathcal B}$ which contains ${\mathcal A}.$
Since $\mathbf{G}$ is of finite order, we know that, for every $\tau>0$, there exist a number $n_{\tau}\in\mathbb{N}$
and a continuous mapping $C_{n_{\tau}} : [0,\tau)\rightarrow L(E)$  such that (\ref{utf-88-second}) holds true.
Define
\[
S_{n_{\tau}+1}(t):=\begin{pmatrix}
\int_0^tC_{n_{\tau}}(s)\,ds &\int_0^t(t-s)C_{n_{\tau}}(s)\,ds\\
C_{n_{\tau}}(t)-g_{n_{\tau}+1}(t)C\; &\int_0^tC_{n_{\tau}}(s)\,ds\end{pmatrix},
\;\;0\leq t<\tau .
\]
Then $S_{n_{\tau}+1} : [0,\tau)\rightarrow L(E \oplus E)$ is continuous and
\[
{\mathcal G}(\varphi)=(-1)^{n_{\tau}+1}\int \limits^{\infty}_{0}\varphi^{(n_{\tau}+1)}(t)S_{n_{\tau}+1}(t)\, dt,\quad \varphi \in {\mathcal D}_{(-\tau ,\tau)}.
\]
This immediately implies that $(S_{n_{\tau}+1}(t))_{t\in [0,\tau)}$
is an $(n_{\tau}+1)$-times integrated ${\mathcal C}$-integrated semigroup, and that $(S_{n_{\tau}+1}(t))_{t\in [0,\tau)}$
is an $(n_{\tau}+1)$-times integrated ${\mathcal C}$-integrated existence family with a subgenerator ${\mathcal B}.$ Due to Lemma \ref{2.1.1.11-mlo-ins}, we have that $(C_{n_{\tau}}(t))_{t\in [0,\tau)}$
is an $n_{\tau}$-times integrated ${\mathcal C}$-times integrated cosine function so that it remains to be proved that $(C_{n_{\tau}}(t))_{t\in [0,\tau)}$
is an $n_{\tau}$-times integrated ${\mathcal C}$-cosine existence family with subgenerator ${\mathbf A},$ i.e., that
$ (\int_0^t(t-s)C_{n_{\tau}}(s)x\, ds, C_{n_{\tau}}(t)x-g_{n_{\tau}+1}  (t)Cx)\in {\mathbf A}$ for all $t\in [0,\tau)$  and $x\in E.$ This is equivalent to say that
\[
\Biggl(\binom{\int_0^t(t-s)C_{n_{\tau}}(s)x\,ds}{0},\binom{0}{\int_0^t(t-s)C_{n_{\tau}}(s)x\,ds}\Biggr)\in {\mathcal B},\quad x\in E,\ t\in [0,\tau),
\]
which simply follows from the inclusion (\ref{phase-space-komutiranje}) and the fact that $(S_{n_{\tau}+1}(t))_{t\in [0,\tau)}$
is an $(n_{\tau}+1)$-times integrated ${\mathcal C}$-integrated existence family with a subgenerator ${\mathcal B}.$ The proof of the theorem is thereby complete.
\end{proof}

\begin{rem}\label{ultra-trsic}
\begin{itemize}
\item[(i)] If ${\mathbf A}=A$ is single-valued, then ${\mathcal A}$ is single-valued, as well. If so, then $(S_{n_{\tau}+1}(t))_{t\in [0,\tau)}$
is an $(n_{\tau}+1)$-times integrated ${\mathcal C}$-integrated semigroup with a subgenerator ${\mathcal A},$ which implies by Lemma \ref{2.1.1.11-mlo-ins}(ii) that $(C_{n_{\tau}}(t))_{t\in [0,\tau)}$ is an $n_{\tau}$-times integrated $C$-cosine function with a subgenerator $A.$ 
\item[(ii)] If the assumptions of Theorem \ref{lokal-int-C-secondinjo} hold, then ${\mathbf G}(\varphi){\mathbf G}(\psi)={\mathbf G}(\psi){\mathbf G}(\varphi),$  $\varphi,\ \psi \in {\mathcal D}$ (in the Banach space setting, this gives the affirmative answer to the question raised on p. 769 of \cite{kolje}). As a simple consequence, we have that, for every $\psi \in {\mathcal D}$, we have
${\mathbf G}(\psi ) {\mathbf A}\subseteq {\mathbf A}{\mathbf G}(\psi )$
and
${\mathbf G}^{-1}(\psi ) {\mathbf A}\subseteq {\mathbf A}{\mathbf G}^{-1}(\psi ).$
\end{itemize}
\end{rem}

Next theorem is a direct consequence of Lemma \ref{2.1.1.11-mlo-ins} and Proposition \ref{5cp-deg}. This theorem is an extension of \cite[Theorem 3.2.5(iii)]{knjigah} and analogue of \cite[Theorem 4.8]{polugrupe-neijektivne-cds} for degenerate differential equations of second order.

\begin{thm}\label{lokal-int-C-prim-second}
Assume that there exists a sequence $((p_k,\tau_k))_{k\in \mathbb{N}_{0}}$ in $\mathbb{N}_{0} \times (0,\infty)$
such that $\lim_{k\rightarrow \infty}\tau_{k}=\infty ,$ $(p_k)_{k\in \mathbb{N}_{0}}$ and
$(\tau_k)_{k\in \mathbb{N}_{0}}$ are strictly increasing,
as well as that for each
$k\in \mathbb{N}_{0}$
there exists a local $p_{k}$-times integrated
$C$-cosine function $(C_{p_{k}}(t))_{t\in [0,\tau_{k})}$ on $E$
satisfying that
\begin{align}\label{zavezi-dot-second}
C_{p_{m}}(t)x=\bigl(g_{p_{m}-p_{k}}\ast_{0} C_{p_{k}}(\cdot)x \bigr)(t ),\quad x\in E,\ t\in [0,\tau_{k}),
\end{align}
provided $k<m.$
Define
$$
{\mathbf G}(\varphi)x:=(-1)^{p_{k}}\int \limits^{\infty}_{0}\varphi^{(p_{k})}(t)C_{p_{k}}(t)x\, dt,\quad \varphi \in {\mathcal D}_{(-\infty ,\tau_{k})},\ x\in E,\ k\in \mathbb{N}_{0}.
$$
Then $
{\mathbf G}$ is well-defined and $
{\mathbf G}$ is a pre-(C-DCF).
\end{thm}

As in the case of degenerate $C$-distribution semigroups, we have the following remarks and comments on Theorem \ref{lokal-int-C-prim-second}.

\begin{rem}\label{Banach}
\begin{itemize}
\item[(i)] Let ${\mathbf A}_{k}$ be the integral generator of $(C_{p_{k}}(t))_{t\in [0,\tau_{k})}$ ($k\in \mathbb{N}_{0}$).
Then ${\mathbf A}_{k}\subseteq {\mathbf A}_{m}$ for $k>m$ and  $\bigcap_{{k\in {\mathbb N}_{0}}}{\mathbf A}_{k}\subseteq {\mathbf A},$ where ${\mathbf A}$ is the integral generator of ${\mathbf G}.$ Even in the case that $C=I,$ $\bigcap_{{k\in {\mathbb N}_{0}}}{\mathbf A}_{k}$ can be a proper subset of ${\mathbf A}.$
\item[(ii)] Suppose that
${\mathbf A}$ is a subgenerator of $(C_{p_{k}}(t))_{t\in [0,\tau_{k})}$ for all
$k\in \mathbb{N}_{0}.$ Then (\ref{zavezi-dot-second}) automatically holds.
\item[(iii)] If $C=I,$ then it suffices to suppose that there exists an MLO ${\mathbf A}$ subgenerating a local $p$-times integrated
cosine function $(C_{p}(t))_{t\in [0,\tau)}$ for some $p\in {\mathbb N}$ and $\tau>0$  (\cite{catania}).
\end{itemize}
\end{rem}

Proposition \ref{5cp-deg} enables us to simply introduce the notion of an exponential pre-(C-DCF) in $E$ (exponential pre-(C-UDCF)  of $\ast$-class  in $E$):

\begin{defn}\label{qdfn}
Let ${\mathbf G}$ be a pre-(C-DCF) (pre-(C-UDCF) of $\ast$-class).
Then $\mathbf{G}$ is said to be an exponential
pre-(C-DCF) (pre-(C-UDCF) of $\ast$-class) iff there exists $\omega\in\mathbb{R}$ such that $e^{-\omega t}\mathcal{G}
\in\mathcal{S}'(L(E\oplus E))$ ($e^{-\omega t}\mathcal{G}
\in\mathcal{S}'^{\ast}(L(E\oplus E))$).
We use the shorthand pre-(C-EDCF) (pre-(C-EUDCF) of $\ast$-class) to denote an exponential
pre-(C-DCF) (pre-(C-UDCF) of $\ast$-class).
\end{defn}

It can be simply verified that a pre-(C-DCF) (pre-(C-UDCF) of $\ast$-class) ${\mathbf G}$ is exponential iff
there exists $\omega\in\mathbb{R}$ such that $e^{-\omega t}\mathbf{G}^{-1}
\in\mathcal{S}'(L(E))$ ($e^{-\omega t}\mathbf{G}^{-1}
\in\mathcal{S}'^{\ast}(L(E))$).\\
\indent
Let $\alpha\in(0,\infty)\setminus \mathbb{N}$, $f\in\mathcal{S}$ and $n=\lceil\alpha\rceil$. Let us recall
that the Weyl fractional derivative $W^{\alpha}_+$ of order $\alpha$
is defined by
\begin{align*}
W^{\alpha}_+f(t):=\frac{(-1)^n}{\Gamma(n-\alpha)}\frac{d^n}{dt^n}\int\limits^{\infty}_t(s-t)^{n-\alpha-1}f(s)\,ds,
\;t\in\mathbb{R}.
\end{align*}
If $\alpha=n\in\mathbb{N}_{0}$, then we set $W^n_+:=(-1)^n\frac{d^n}{dt^n}.$
\begin{thm}\label{miana-second-degenerate}
Assume that $\alpha \geq 0$ and that ${\mathbf A}$ is the integral generator of a global $\alpha$-times integrated $C$-cosine function $(C_{\alpha}(t))_{t\geq 0}$ on $E.$
Set
\begin{align*}
\mathbf{G}_{\alpha}(\varphi)x:=\int^{\infty}_0W^{\alpha}_+\varphi(t)C_{\alpha}(t)x\,dt,\quad x\in E,\ \varphi\in\mathcal{D}.
\end{align*}
Then ${\mathbf G}$ is a pre-(C-DCF) whose integral generator contains ${\mathbf A}.$ Furthermore, if $(C_{\alpha}(t))_{t\geq 0}$
is exponentially equicontinuous, then ${\mathbf G}$ is exponential.
\end{thm}
\begin{proof}
Note that if ${\mathcal A}$ is the integral generator of a global $\alpha$-times integrated $C$-semigroup $(S_{\alpha}(t))_{t\geq 0}$ on $E$ and
$\mathcal{G}_{\alpha}(\varphi)x:=\int^{\infty}_0W^{\alpha}_+\varphi(t)S_{\alpha}(t)x\,dt,$ $x\in E,$ $\varphi\in\mathcal{D}$
then ${\mathcal G}$ is a pre-(C-DS) whose integral generator contains ${\mathcal A}.$ Then by Lemma \ref{2.1.1.11-mlo-ins} and Lemma \ref{kisinskija-second} we have $\mathcal{B}:=\bigl(\begin{smallmatrix}0&I\\ {\mathcal A} &0\end{smallmatrix}\bigr)$ is a subgenerator of an $(\alpha+1)$-times integrated $C$-semigroup $(S_{\alpha+1}(t))_{t\geq0}$, hence ${\mathcal A}$ is a subgenerator of $\alpha$-times integrated $C$-cosine function $(C_{\alpha}(t))_{t\geq0}$ on $E$. Then by Proposition \ref{5cp-deg} we obtain the statement of the theorem.
\end{proof}

\begin{rem}\label{xcvbnm-prim-ght}
It is clear that ${\mathbf G}(\cdot)\equiv 0$ is a degenerate pre-distribution cosine function with the generator ${\mathcal A}\equiv E\times E,$ as well as that, for every $\tau>0$ and for every integer $n_{\tau}\in\mathbb{N},$ there exists only one local $n_{\tau}$-times integrated cosine function $(C_{n_{\tau}}(t)\equiv 0)_{t\in [0,\tau)}$ satisfying
(\ref{utf-88-second}). Then condition (B)' holds and condition (A)' does not hold here.
Designate by ${\mathbf A}_{\tau}$ the integral generator of $(C_{n_{\tau}}(t)\equiv 0)_{t\in [0,\tau)}.$
Then ${\mathrm A}_{\tau}=\{0\} \times E$ is strictly contained in the integral generator ${\mathbf A}$
of ${\mathbf G}.$
Furthermore, if
$C\neq 0,$ then
there do not exist numbers
$\tau>0$ and $n_{\tau}\in\mathbb{N}$ such that ${\mathbf A}_{\tau}$ generates (subgenerates) a local
$n_{\tau}$-times integrated $C$-cosine function.
\end{rem}

The notion of a $ q$-exponential pre-(C-DCF) (pre-(C-UDCF) of $\ast$-class) can be also introduced and further analyzed. For the sake of brevity, we shall skip all related details concerning this topic here.

We close this section with the observation that the assertions of
\cite[Theorem 3.6.13, Theorem 3.6.14]{knjigah} can be simply reformulated for non-degenerate ultradistribution sines in locally convex spaces. For more details concerning the semigroup case, the reader may consult \cite{C-ultra-dvojka}.

\section[Examples and applications]{Examples and applications}

First of all, we would like to draw the readers' attention on some instructive examples of non-degenerate ultradistribution sines in Fr\'echet spaces.

\begin{example}\label{nedegen-sinusi}
\begin{itemize}
\item[(i)]
Set $E:=\{f\in C^{\infty}([0,\infty)) : \lim
_{x\rightarrow +\infty}f^{(k)}(x)=0\mbox{ for all }k\in {{\mathbb
N}_{0}}\}.$ Equipped with the family of norms $||f||_{k}:=\sum_{j=0}^{k}\sup_{x\geq
0}|f^{(j)}(x)|,$ $f\in E$ ($k\in {{\mathbb N}_{0}}$), $E$ becomes a Fr\'echet space.
Suppose $c_{0}>0,$ $\beta>0,$ $s>1$ and $M_{p}:=p!^{s}.$ Define the operator
$A$ by $D(A):=\{u\in E : c_{0}u^{\prime}(0)=\beta u(0)\}$ and
$Au:=c_{0}u^{\prime \prime}.$ One can prove that, for every two sufficiently small number $\varepsilon>0,\ \varepsilon'>0$ and for every integer $k\in {\mathbb N}_{0},$ there exist constants $c(\varepsilon,\varepsilon')>0$ and $c(k,\varepsilon,\varepsilon')>0$ such that
\begin{equation}\label{faithdj}
\Bigl \|  (\lambda-A)^{-1}f \Bigr \|_{k}\leq c(k,\varepsilon,\varepsilon')e^{c(\varepsilon,\varepsilon')|\lambda|^{\varepsilon'}}\bigl \|f \bigr \|_{k},\quad f\in E,\ \lambda \in \Sigma_{\pi -\varepsilon}.
\end{equation}
Set for $\bar{a}>0$,
$$
{\mathbf G}(\varphi)f:=(-i)\int\limits_{\bar{a}-i\infty}^{\bar{a}
+i\infty}\lambda \hat{\varphi}(\lambda) \bigl(\lambda^{2}-A\bigr)^{-1}f\,d\lambda,\quad f\in E,\ \varphi\in {\mathcal D}^{(M_{p})}.
$$
Then ${\mathbf G}$ is an exponential pre-(EUDCF) of $(M_{p})$-class,
$\mathbf{G}(\varphi)A\subseteq A{\mathbf G}(\varphi),$ $\varphi \in {\mathcal D}^{(M_{p})}$ and $A\mathbf{G}(\varphi)f=\mathbf{G}\bigl(\varphi^{\prime \prime}\bigr)f+\varphi^{\prime}(0)f,$ $f\in E,\ \varphi \in {\mathcal D}^{\ast};$
cf.  Remark \ref{filipa-1-frim}.

Now we will prove that the condition
$(CCF_{2})$ holds. Suppose that $ {\mathbf G}(\varphi
)f+{\mathbf G}^{-1}(\varphi )g=0,$  $ \varphi \in {\mathcal D}^{(M_{p})}_{0}$ for some functions $f,\ g\in E,$ i.e.,
\begin{equation}\label{neqneq}
\int\limits_{\bar{a}-i\infty}^{\bar{a}
+i\infty}\lambda \hat{\varphi}(\lambda) \bigl(\lambda^{2}-A\bigr)^{-1}f\,d\lambda +\int\limits_{\bar{a}-i\infty}^{\bar{a}
+i\infty}\hat{\varphi}(\lambda) \bigl(\lambda^{2}-A\bigr)^{-1}g\,d\lambda=0,\quad \varphi \in {\mathcal D}^{(M_{p})}_{0}.
\end{equation}

Let us recall that any complex number $\lambda \in {\mathbb C} \setminus (-\infty,0]$ belongs to
$\rho (A),$ and that
\begin{align*}
& \bigl( \lambda -A \bigr)^{-1}f(x)= \frac{1}{2{\sqrt{c_{0}\lambda}}}\Biggl[ \int ^{x}_{0}e^{-\sqrt{\lambda/c_{0}}(x-s)}f(s)\, ds+\int
^{\infty}_{x}e^{\sqrt{\lambda/c_{0}}(x-s)}f(s)\, ds\Biggr]
\\ &+\frac{1}{c_{0}}\Biggl[
\frac{c_{0}\sqrt{\lambda/c_{0}}-\beta}{c_{0}\sqrt{\lambda/c_{0}}+\beta}\frac{1}{2\sqrt{\lambda/c_{0}}}\int^{\infty}_{0}e^{-\sqrt{\lambda/c_{0}}
s}f(s)\, ds\Biggr]e^{-\sqrt{\lambda/c_{0}} x},\ x\geq 0,\ f\in E,
\end{align*}
as well as that
$$
(-i)\int\limits_{\bar{a}-i\infty}^{\bar{a}
+i\infty}\hat{\varphi}(\lambda)e^{-\lambda t}\, d\lambda=\varphi (t),\quad \varphi\in {\mathcal D}^{(M_{p})},\ t\geq 0.
$$
It can be calculated that (see \cite[Example 4.5]{C-ultra-dvojka})
$$
\notag {\mathcal G}(\varphi)f(x)=
\frac{1}{2\sqrt{\pi c_{0}}}\int^{\infty}_{0}\varphi(t)\Biggl[ t^{(-1)/2}\int \limits^{x}_{0}e^{-s^{2}/4c_{0}t}f(x-s)\, ds \Biggr]\, dt+$$
$$+\frac{1}{2\sqrt{\pi c_{0}}}\int^{\infty}_{0}\varphi(t)\Biggl[ t^{(-1)/2}\int \limits^{\infty}_{0}e^{-s^{2}/4c_{0}t}f(x+s)\, ds \Biggr]\, dt+
$$ $$+\frac{1}{2\sqrt{c_{0}}} \int^{\infty}_{0}\varphi(t) \Biggl[ \frac{1}{2\pi i}\int\limits_{\bar{a}-i\infty}^{\bar{a}
+i\infty} e^{\lambda t}f(\lambda;x)\, d\lambda \Biggr]\, dt,\ x\geq 0,\ f\in E,\ \varphi\in {\mathcal D}^{(M_{p})}_{0},$$
where
$$
f(\lambda;x):=\Biggl[ e^{-\sqrt{\lambda  c_{0}^{-1}}x}
\frac{c_{0}\sqrt{\lambda c_{0}^{-1}}-\beta}{c_{0}\sqrt{\lambda c_{0}^{-1}}+\beta}\frac{1}{\sqrt{\lambda}}\int^{\infty}_{0}e^{-\sqrt{\lambda c_{0}^{-1}} v}f(v)\, dv \Biggr]
$$
for every $x> 0$ and for every $\lambda \in {\mathbb C}$ with $\Re \lambda =\bar{a}.$
The integrals in (\ref{neqneq}) are equal to zero. The statement that second integral is equal to zero is equivalent to ${\mathcal G}(\varphi)g=0.$ It is shown in \cite[Example 4.5]{C-ultra-dvojka} that it must be $g=0$.
Now, let the first integral be zero. By \cite[Corollary 1.6.6]{a43} and again using the same argumentation as in the previous case when ${\mathcal G}(\varphi)g=0$ we deduce that it must $f=0$.


This implies (see e.g. \cite[Theorem 3.6.14]{knjigah} for the Banach space case) that the abstract Cauchy problem
$$
(ACP_2):\left\{
\begin{array}{l}
u\in C^{\infty}([0,\infty):E)\cap C([0,\infty):[D(A)]),\\[0.1cm]
u_{tt}(t,x)=c_{0}u_{xx}(t,x),\;t\geq 0,\ x\geq 0,\\ \ u(0,x)=u_{0}(x),\;u_{t}(0,x)=u_{1}(x),\ x\geq 0
\end{array}
\right.
$$
has a unique solution for any $u_{0},\ u_{1}\in E^{(M_{p})}(A),$ where $E^{(M_{p})}(A)$ is the abstract Beurling space consisting of those functions $f\in E$ satisfying that, for every $h>0$ and $n\in {\mathbb N},$ we have $\sup_{p\in {\mathbb N}_{0}}(h^{p}\| f^{(2p)}\|_{n}/M_{p})<\infty;$
furthermore, for every compact set $K\subseteq[0,\infty)$ and for every $n\in {\mathbb N}$ and $h>0$, the solution $u$ of $(ACP_2)$ satisfies
$$
\sup_{t\in K,\;p\in\mathbb{N}_0}\frac{h^p}{M_{p}}\Biggl(\Big\|\frac{d^p}{dt^p}u(t)\Big\|_{n}+\Big\|\frac{d^p}{dt^{p+1}}u(t)\Big\|_{n}\Biggr)<\infty.
$$
Suppose now that $P(z)$ is a non-constant complex polynomial of degree $k\in {\mathbb N}$ such that there exist positive real numbers $a,\ b>0$ such that, for every $\lambda \in {\mathbb C}$ with $\Re \lambda>a,$ all the zeroes of polynomial $z\mapsto P(z)-\lambda,$ $z\in {\mathbb C}$ belong to ${\mathbb C} \setminus (-\infty,0].$ Let $\bar{a}>a.$ Then it can be easily seen that, for every two sufficiently small number $\varepsilon>0,\ \varepsilon'>0$ and for every integer $k\in {\mathbb N}_{0},$ there exist constants $c(\varepsilon,\varepsilon')>0$ and $c(k,\varepsilon,\varepsilon')>0$ such that
\begin{equation}\label{faithdj-ikngh}
\Bigl \|  \bigl(\lambda-P(A)\bigr)^{-1}f \Bigr \|_{k}\leq c(k,\varepsilon,\varepsilon')e^{c(\varepsilon,\varepsilon')|\lambda|^{\varepsilon'}}\bigl \|f \bigr \|_{k},\quad f\in E,\ \Re \lambda >\bar{a}.
\end{equation}
Set
$$
{\mathbf G}_{P}(\varphi)f:=(-i)\int\limits_{\bar{a}-i\infty}^{\bar{a}
+i\infty}\lambda \hat{\varphi}(\lambda) \bigl(\lambda^{2}-P(A)\bigr)^{-1}f\,d\lambda,\quad f\in E,\ \varphi\in {\mathcal D}^{(M_{p})}.
$$
Then ${\mathbf G}_{P}$ is an exponential pre-(EUDCF) of $(M_{p})$-class, and it is open question whether the condition $(CCF_{2})$ holds for
${\mathbf G}_{P},$ in general.\\

\item[(ii)] In this part, we use the notation from \cite[Chapter 8]{a43}. Let $p\in[1,\infty)$, $m>0$, $\rho\in[0,1]$, $r>0$, and let $a\in S_{\rho,0}^m$ satisfies $(H_r).$ Suppose that
$E=L^p(\mathbb{R}^n)$ or $E=C_0(\mathbb{R}^n)$
(in the second case, we assume $p=\infty$), $0\leq l\leq n,$ $A:=\mathrm{Op}_E(a)$
and that the following inequality
\begin{equation}\label{mh}
n\Bigl|\frac{1}{2}-\frac{1}{p}\Bigr|\Bigl(\frac{m-r-\rho+1}{r}\Bigr)<1
\end{equation}
holds.
Let su recall that if $a(\cdot)$ is an elliptic polynomial of order $m$, then \eqref{mh} holds with $m=r$ and $\rho=1.$
Suppose that there exists a sequence $(M_p)$ satisfying (M.1), (M.2) and (M.3$'$), as well as that $a(\mathbb{R}^n)\cap\Lambda_{l,\zeta,\eta}^{2}=\emptyset$ for some constants
$l\geq 1$, $\zeta>0$ and $\eta\in\mathbb{R}.$ Here
$$
\Lambda_{l,\zeta,\eta}=\bigl\{\lambda\in\mathbb{C}:\Re\lambda\geq\zeta M(l|\Im\lambda|)+\eta \bigr\} \mbox{ and } \Lambda_{l,\zeta,\eta}^{2}=\bigl\{\lambda^{2} : \lambda \in \Lambda_{l,\zeta,\eta}\bigr\}.
$$
Put ${{\mathbb
N_{0}^{l}}}:=\{\eta \in {\mathbb N_{0}^{n}} : \eta_{l+1}=\cdot \cdot
\cdot =\eta_{n}=0\}$ and
$ E_{l}:=\{ f\in E : f^{(\eta )} \in E \mbox{ for all }\eta \in
{{\mathbb N_{0}^{l}}}\}.$ Then the calibration
$(q_{\eta}(f):=||f^{(\eta )}||_{E},\ f\in E_{l};\ \eta \in {{\mathbb
N_{0}^{l}}})$ induces a Fr\' echet topology on $E_{l}$ (\cite{x263}).
Define the operator $A_{l}$ on $E_{l}$ by $D(A_{l}):=\{f\in E_{l} : \mathrm{Op}_E(a)f\in E_{l}\}$ and $A_{l}f:=\mathrm{Op}_E(a)f$
($f\in D(A_{l})$). Then we know that there exist numbers $\eta'\geq\eta ,$ $N\in {\mathbb N}$ and $M\geq 1$ such that $\Lambda_{l,\zeta,\eta'}^{2} \subseteq \rho(A_{l})$ and that for each $\eta \in {{\mathbb
N_{0}^{l}}}$ we have
$$
q_{\eta}\Bigl(R\bigl(\lambda:A_{l}\bigr)f\Bigr) \leq  M\bigl(1+|\lambda|\bigr)^{N}q_{\eta}(f),\quad \lambda \in \Lambda_{l,\zeta,\eta'}^{2},\ f\in E_{l}.
$$
Keeping in mind Theorem \ref{herbs-second} and Remark \ref{filipa-1-frim}, we get that $A_{l}$ generates an ultradistribution sine of $(M_{p})$-class in $E_{l}.$
\end{itemize}
\end{example}

\begin{example}\label{degener-c-obratne-prim}
Multiplication operators in $L^{p}$-spaces generating degenerate locally integrated cosine functions can be simply constructed following the method proposed in \cite[Example 3.2.11]{FKP} and \cite[Example 3.4.44]{knjigah}. These examples can serve for construction of non-exponential pre-(DCF)'s in Banach spaces by Theorem \ref{miana-second-degenerate}.
\end{example}

\begin{example}\label{degener-c-obratne}
Suppose that $(E,\| \cdot \|)$ is a Banach space.
In \cite[Chapter III]{faviniyagi}, A. Favini and A. Yagi have considered  the multivalued linear operators satisfying the following condition:
\begin{itemize}
\item[(PW)]
There exist finite constants $c,\ M>0$ and $\beta \in (0,1]$ such that\index{condition!(PW)}
$$
\Psi:=\Psi_{c}:=\Bigl\{ \lambda \in {\mathbb C} : \Re \lambda \geq -c\bigl( |\Im \lambda| +1 \bigr) \Bigr\} \subseteq \rho({\mathcal A})
$$
and
$$
\| R(\lambda : {\mathcal A})\| \leq M\bigl( 1+|\lambda|\bigr)^{-\beta},\quad \lambda \in \Psi.
$$
\end{itemize}
If (PW) holds, then it can be simply proved that there exists a continuous linear operator $C$ such that ${\mathcal A}^{2}$ is a subgenerator of a global once integrated $C$-cosine function that is not exponentially bounded, in general (\cite{FKP}). This example and Theorem \ref{miana-second-degenerate} can be used for construction of non-exponential pre-(C-DCF)'s in Banach spaces.
\end{example}

\section[Appendix]{Appendix: Fractionally integrated $C$-cosine functions}\label{maxx}
Here we list some definitions and statements that are used previously in the paper (most of them are already known). We refer to \cite{catania} for the definition of (local, if $\tau<\infty$) $\alpha$-times integrated $C$-cosine functions.\\
\indent Let $g_{\alpha}(t)=\frac{t^{\alpha-1}}{\Gamma({\alpha})}$, for $t>0$. Recall that a strongly continuous operator family  $((C_{\alpha})(t))_{t\in(0,\tau)}\subseteq L(E)$ is called a (local, if $\tau<\infty$) $\alpha$-times integrated $C$-cosine function iff the following holds: $(i)\,\, C_\alpha(t)C=CC_\alpha(t)$, $t\in [0,\tau),$ and \\ For all $x\in E$ and $t,\ s\in [0,\tau)$ with $t+s\in [0,\tau),$
\begin{equation}\label{composition}
\begin{aligned}
2C_\alpha(t)C_\alpha(s)x&=\Biggl(\int\limits^{t+s}_t-\int\limits^s_0\Biggr)g_{\alpha}(t+s-r)C_\alpha(r)Cx\,dr
\\[-2ex]
\!+\!\int\limits^t_{t-s}\!\! g_{\alpha}(r&-t+s)C_\alpha(r)Cx\,dr\!+\!\int\limits^s_0\!\! g_{\alpha}(r+t-s)C_\alpha(r)Cx\,dr,\;\;t\geq s;
\\
2C_\alpha(t)C_\alpha(s)x&=\Biggl(\int\limits^{t+s}_s-\int\limits^t_0\Biggr)g_{\alpha}(t+s-r)C_\alpha(r)Cx\,dr
\\[-2ex]
\!+\!\int\limits^s_{s-t}\!\! g_{\alpha}(r&+t-s)C_\alpha(r)Cx\,dr\!+\!\int\limits^t_0\!\! g_{\alpha}(r-t+s)C_\alpha(r)Cx\,dr,\;\;t<s.
\end{aligned}
\end{equation}


We refer to \cite{catania} for a (local) $C$-regularized semigroup, resp., (local) $C$-regularized cosine function.

Let $0<\alpha<\infty$. In
the case $\tau=\infty ,$ $(S_{\alpha}(t))_{t\geq 0}$ is said to be
exponentially equicontinuous\index{$(a,k)$-regularized $C$-resolvent family!exponentially equicontinuous} (equicontinuous\index{$(a,k)$-regularized $C$-resolvent family!equicontinuous}) iff there exists
$\omega \in {\mathbb R}$ ($\omega =0$) such that the family $\{
e^{-\omega t} S_{\alpha}(t) : t\geq 0\}$ is equicontinuous. The above notion can be simply understood for the class of fractionally integrated $C$-cosine functions.
The integral generator $\hat{{\mathcal A}}$ of $(S_{\alpha}(t))_{t\in [0,\tau)},$ resp. $(C_{\alpha}(t))_{t\in [0,\tau)},$ is defined by graph
\[
\hat{{\mathcal A}}:=\Biggl\{(x,y)\in E\times E:S_\alpha(t)x-g_{\alpha+1} (t)Cx=\int\limits_0^tS_\alpha(s)y\,ds,\; t\in [0,\tau) \Biggr\}, resp.,
\]

\[
\hat{{\mathcal A}}:=\Biggl\{(x,y)\in E\times E: C_\alpha(t)x-g_{\alpha+1} (t)Cx=\int\limits_0^t(t-s)C_\alpha(s)y\,ds,\; t\in [0,\tau) \Biggr\}.
\]
The integral generator $\hat{{\mathcal A}}$ of $(S_{\alpha}(t))_{t\in [0,\tau)},$ resp. $(C_{\alpha}(t))_{t\in [0,\tau)},$ is a closed MLO in $E.$
Furthermore, $\hat{{\mathcal A}}\subseteq C^{-1}\hat{{\mathcal A}}C$ in the MLO sense, with the equality in the case that
the operator $C$ is injective.

By a subgenerator of $(S_{\alpha}(t))_{t\in [0,\tau)},$ resp. $(C_{\alpha}(t))_{t\in [0,\tau)},$ we mean any MLO ${\mathcal A}$ in $E$ satisfying the following two conditions:
\begin{itemize}
\item[(A)] $S_{\alpha}(t)x-g_{\alpha+1} (t)Cx=\int_0^tS_{\alpha}(s)y\,ds,\mbox{ whenever }t\in [0,\tau)\mbox{ and }y\in {\mathcal A}x;$
\item[(B)] For all $x\in E$ and $t\in [0,\tau),$  we have $\int^{t}_{0}S_{\alpha}(s)x\, ds \in D({\mathcal A})$ and
$S_{\alpha}(t)x-g_{\alpha+1}(t)Cx\in {\mathcal A}\int_0^t S_{\alpha}(s)x\,ds,$
\end{itemize}
resp.,
\begin{itemize}
\item[(A)'] $C_{\alpha}(t)x-g_{\alpha+1} (t)Cx=\int_0^t (t-s)C_{\alpha}(s)y\,ds,\mbox{ whenever }t\in [0,\tau)\mbox{ and }y\in {\mathcal A}x;$
\item[(B)'] For all $x\in E$ and $t\in [0,\tau),$  we have $\int^{t}_{0}(t-s)C_{\alpha}(s)x\, ds \in D({\mathcal A})$ and
$C_{\alpha}(t)x-g_{\alpha+1}(t)Cx\in {\mathcal A}\int_0^t (t-s)C_{\alpha}(s)x\,ds .$
\end{itemize}
If $(S_{\alpha}^{1}(t))_{t\in [0,\tau)}\subseteq  L(E),$ resp. $(S_{\alpha}^{2}(t))_{t\in [0,\tau)}\subseteq  L(E)$
($(C_{\alpha}^{1}(t))_{t\in [0,\tau)}\subseteq  L(E),$ resp. $(C_{\alpha}^{2}(t))_{t\in [0,\tau)}\subseteq  L(E)$),
is strongly continuous and satisfies only (B), resp. (A) ((B)', resp. (A)'),
then we say that
$(S_{\alpha}^{1}(t))_{t\in [0,\tau)},$ resp. $(S_{\alpha}^{2}(t))_{t\in [0,\tau)}$ ($(C_{\alpha}^{1}(t))_{t\in [0,\tau)},$ resp. $(C_{\alpha}^{2}(t))_{t\in [0,\tau)}$),
is an $\alpha$-times integrated $C$-existence family with a subgenerator ${\mathcal A},$ resp.,
$\alpha$-times integrated $C$-uniqueness family with a subgenerator ${\mathcal A}$ ($\alpha$-times integrated $C$-cosine existence family with a subgenerator ${\mathcal A},$ resp.,
$\alpha$-times integrated $C$-cosine uniqueness family with a subgenerator ${\mathcal A}$).

By $\chi(S_{\alpha}),$ resp., $\chi(C_{\alpha}),$
we denote the set consisting of all subgenerators of $(S_{\alpha}(t))_{t\in [0,\tau)},$ resp., $(C_{\alpha}(t))_{t\in [0,\tau)}.$
It is well known (see \cite{knjigah}, \cite{catania}) that any of the sets $\chi(S_{\alpha})$ and  $\chi(C_{\alpha})$
can have infinitely many elements; if ${\mathcal A}\in \chi(S_{\alpha})$, resp. ${\mathcal A}\in \chi(C_{\alpha})$, then
${\mathcal A}\subseteq \hat{{\mathcal A}}.$ In general, the set $\chi(S_{\alpha})$ can be empty and the integral generator of $(S_{\alpha}(t))_{t\in [0,\tau)}$
need not be a subgenerator of $(S_{\alpha}(t))_{t\in [0,\tau)}$ in the case that $\tau <\infty;$ the same holds for fractionally integrated
$C$-cosine functions.
In global case, the integral generator $\hat{{\mathcal A}}$ of $(S_{\alpha}(t))_{t\geq 0},$ resp. $(C_{\alpha}(t))_{t\geq 0},$ is always its subgenerator. If ${\mathcal A}$ is a closed subgenerator of $(S_{\alpha}(t))_{t\in [0,\tau)},$ resp. $(C_{\alpha}(t))_{t\geq 0},$ defined locally  or globally, then we know that $C{\mathcal A}\subseteq {\mathcal A}C,$ $\hat{{\mathcal A}}\subseteq C^{-1}{\mathcal A}C$ and that the injectivity of $C$
implies $\hat{{\mathcal A}}= C^{-1}{\mathcal A}C.$ Suppose that $C$ is injective
and ${\mathcal A}$ is an MLO. Then there exists at most one $\alpha$-times integrated $C$-semigroup $(S_\alpha(t))_{t\in [0,\tau)},$
resp. $\alpha$-times integrated $C$-cosine function $(C_{\alpha}(t))_{t\in [0,\tau)},$ which do have ${\mathcal A}$ as a subgenerator
(\cite{catania}).

We need the following results from \cite{catania}.

\begin{lem}\label{2.1.1.11-mlo-ins}(\cite{catania})
Suppose that ${\mathcal A}$ is a closed \emph{MLO} in $E$, $0<\tau\leq\infty,$ $0\leq \alpha < \infty,$
and $(C_{\alpha}(t))_{t\in [0,\tau)}$ is a strongly continuous operator family which commutes with $C.$ Set \[
S_{\alpha+1}(t)=\begin{pmatrix}
\int_0^tC_\alpha(s)\,ds &\int_0^t(t-s)C_{\alpha}(s)\,ds\\
C_\alpha(t)-g_{\alpha+1}(t)C\; &\int_0^tC_\alpha(s)\,ds\end{pmatrix},
\;\;0\leq t<\tau
\]
and ${\mathcal C}(x\ y)^{T}:=(Cx \ Cy)^{T}$ ($x,\ y\in E$).
Then we have:
\begin{itemize}
\item[(i)] The following assertions are equivalent:

\smallskip
 \emph{(a)} {$(C_\alpha(t))_{t\in[0,\tau)}$ is an $\alpha$-times integrated $C$-cosine function on $E$.}

\smallskip
 \emph{(b)} $(S_{\alpha+1}(t))_{t\in[0,\tau)}$ is an $(\alpha+1)$-times integrated $\mathcal{C}$-semigroup \\
$(S_{\alpha+1}(t))_{t\in[0,\tau)}$ on $E\times E.$
\end{itemize}
Suppose that the equivalence relation \emph{(a)} $\Leftrightarrow$ \emph{(b)} in \emph{(i)} holds. Then we have:
\begin{itemize}
\item[(ii)] ${\mathcal A}$ is a subgenerator of $(C_\alpha(t))_{t\in[0,\tau)}$ iff
$\mathcal{B}:=\bigl(\begin{smallmatrix}0&I\\ {\mathcal A} &0\end{smallmatrix}\bigr)$
is a subgenerator of $(S_{\alpha+1}(t))_{t\in[0,\tau)}.$
\item[(iii)] Let ${\hat{\mathcal A}}$ and ${\hat{\mathcal B}}$
be the integral generators of $(C_\alpha(t))_{t\in[0,\tau)}$ and $(S_{\alpha+1}(t))_{t\in[0,\tau)}$,
respectively.  Then the inclusion $\bigl(\begin{smallmatrix}0&I\\ \hat{{\mathcal A}} &0\end{smallmatrix}\bigr) \subseteq {\hat{\mathcal B}}$ holds true. Furthermore, if $(C_{\alpha}(t))_{t\in [0,\tau)}$ is non-degenerate, then $\bigl(\begin{smallmatrix}0&I\\ \hat{{\mathcal A}} &0\end{smallmatrix}\bigr) = {\hat{\mathcal B}}.$
\end{itemize}
\end{lem}

\end{document}